\documentclass[12pt]{article}
\usepackage{amssymb}
\usepackage{amscd}
\usepackage{graphicx}
\usepackage[all]{xy}
\usepackage{color}
\usepackage{latexsym}
\usepackage{amssymb,amsmath,amstext}
\usepackage{epsfig}
\usepackage{graphics}
\usepackage{subfigure}
\usepackage{euscript}
\usepackage{ifthen}
\usepackage{wrapfig}
\usepackage{amsthm}
\usepackage{multicol}
\usepackage{verbatim}
\usepackage[colorlinks=true, urlcolor=blue, linkcolor=blue, citecolor=blue]{hyperref}
\usepackage[margin=1in]{geometry}

\numberwithin{equation}{section}
\theoremstyle{plain}%
\newtheorem{theorem}{Theorem}
\numberwithin{theorem}{section}
\newtheorem{proposition}[theorem]{Proposition}
\newtheorem{example}[theorem]{Example}
\newtheorem{lemma}[theorem]{Lemma}
\newtheorem{corollary}[theorem]{Corollary}

\newtheorem{remark}[theorem]{Remark}
\newtheorem{conjecture}[theorem]{Conjecture}

\newcommand{\C}{\mathbb{C}}

\newcommand{\RR}{\mathbb{R}}
\newcommand{\Z}{\mathbb{Z}}
\newcommand{\PP}{\mathbb{P}}
\newcommand{\TP}{\mathbb{TP}}

\newcommand{\T}{\mathbb{T}}

\newcommand{\R}{\mathbb{R}}

\newcommand{\trp}{\text{trop}}

\newcommand{\Div}{\text{div}}

\allowdisplaybreaks

\date{}
\begin{document}

\title{\bf Tropicalization of Del Pezzo Surfaces}

\author{Qingchun Ren, Kristin Shaw and Bernd Sturmfels}

\maketitle

\begin{abstract} 
\noindent
We determine the tropicalizations of very affine surfaces
over a valued field that are obtained
from del Pezzo surfaces of degree $5, 4$ and $3$
by removing their $(-1)$-curves.
On these tropical surfaces, the boundary divisors
are represented by trees at infinity. These trees
are glued  together according to the
Petersen, Clebsch and Schl\"afli graphs, respectively.
There are $27$ trees on each tropical cubic surface, attached to
a bounded complex with
up to $73$ polygons.
The maximal cones in the
$4$-dimensional moduli~fan reveal two generic types
of such surfaces.
\end{abstract}

\section{Introduction}\label{sec:intro}

A smooth cubic surface $X$ in projective $3$-space $\PP^3$ contains
$27$ lines. These lines are characterized intrinsically as 
the $(-1)$-curves on~$X$, that is, rational curves of self-intersection~$-1$.
The tropicalization of an embedded surface $X$ 
is obtained directly from the cubic polynomial  that defines it
in $ \mathbb{P}^3$.
 The resulting    tropical surfaces are dual to
 regular subdivisions  of the size $3$ tetrahedron.
These  come in many combinatorial types \cite[\S 4.5]{MacStu}.
 If the subdivision is a  unimodular  triangulation
 then the tropical surface is called smooth (cf.~\cite[Prop.~4.5.1]{MacStu}).
 
Alternatively, by removing the $27$ lines from
the cubic surface $X$, we obtain
a very affine surface $X^0$.
In this paper, we study  the  tropicalization of $X^0$, denoted ${\rm trop}(X^0)$,
via the embedding in its intrinsic torus \cite{HKT}.
 This is an invariant of the surface $X$. The
 $(-1)$-curves on $X$ now become visible as $27$ \emph{boundary trees}
 on ${\rm trop}(X^0)$.
This distinguishes our approach from Vigeland's work \cite{Vig}
on the $27$ lines on tropical cubics in $\TP^3$. 
It also highlights an important feature of tropical 
geometry \cite{MikRau}: there are different  tropical models of a single classical  variety,
and the choice of model 
depends on what structure one wants revealed. 

Throughout this paper we work over a field
$K$ of characteristic zero that has  a non-archimedean valuation.
Examples include the Puiseux series $K = \C \{ \!\{ t\} \!\}$ and the
$p$-adic numbers $K = \mathbb{Q}_p$.
We use the term {\em cubic surface} to mean
a marked smooth del Pezzo surface $X$ of degree~$3$.
A {\em tropical cubic surface} is the intrinsic tropicalization 
${\rm trop}(X^0)$ described above.
Likewise,  {\em tropical del Pezzo surface} refers to the 
tropicalization ${\rm trop}(X^0)$ for degree $\geq 4$.
 Here, the adjective
``tropical'' is used solely for brevity, 
instead of the more accurate ``tropicalized''  used in \cite{MacStu}.
We do not consider non-realizable tropical del Pezzo surfaces,
nor tropicalizations of surfaces defined over a field $K$ with positive characteristic.
 
The moduli space of cubic surfaces is four-dimensional,
and its tropical version is the four-dimensional {\em Naruki fan}.
This was constructed combinatorially by Hacking, Keel and Tevelev \cite{HKT},
and it was realized in \cite[\S 6]{RSS} as the tropicalization of a very affine variety 
$\mathcal{Y}^0$, obtained from  the 
Yoshida variety $\mathcal{Y}$ in $\PP^{39}$ by intersecting with $(K^*)^{39}$. The Weyl group
$W(\mathrm{E}_6) $ acts on  $\mathcal{Y}$
 by permuting the $40$ coordinates.
The maximal cones in ${\rm trop}(\mathcal{Y}^0)$ come in
two $W(\mathrm{E}_6) $-orbits.
We here compute the corresponding cubic surfaces:

\begin{theorem} \label{thm:deg3}
There are two generic types of
tropical cubic surfaces. They are contractible and  
characterized at infinity by $27$ metric trees, each having $10$ leaves.
The first type has $73$ bounded cells, $150$ edges, $78$ vertices, 
$135$ cones, $189$ flaps, $216$ rays, and all $27$ trees are trivalent.
The second type has $72$ bounded cells, $148$ edges, $77$ vertices, 
$135$ cones, $186$ flaps, $213$ rays, and three 
of the $27$ trees have a $4$-valent node.
\hfill (For more data see Table \ref{tab:cubicsurf}.)
\end{theorem}

Here, by {\em cones} and {\em flaps} we mean
unbounded $2$-dimensional polyhedra
that are affinely isomorphic to $\mathbb{R}_{\geq 0}^2$
and $ [0,1] \times \mathbb{R}_{\geq 0} $ respectively.
The {\em characterization at infinity} is analogous to that  
for tropical planes  in \cite{HJJS}.
Indeed, by \cite[Theorem~4.4]{HJJS},
every tropical plane $L$ in $\mathbb{TP}^{n-1}$
is given by an arrangement of $n$ boundary trees,
each having $n-1$ leaves, and $L$ is uniquely determined by this
arrangement. Viewed intrinsically, $L$ is the tropicalization 
of a very affine surface, namely
the complement of $n$ lines in $\mathbb{P}^2$.
Theorem~\ref{thm:deg3} offers the analogous
characterization for the tropicalization of 
the complement of the
$27$ lines on a cubic surface.

\smallskip

Tropical geometry has undergone an explosive development 
during the past decade. To the outside observer,
the literature is full of conflicting definitions and diverging approaches.
The text books  \cite{MacStu, MikRau} offer some help,
but they each stress one particular point of view.

An important feature of the present paper is its focus on the
unity of tropical geometry.  We shall
develop three different techniques for  computing tropical del Pezzo surfaces:
\begin{itemize}
\item Cox ideals, as explained in Section 2;
\item fan structures on moduli spaces, as explained in Section 3;
\item tropical  modifications, as explained in Section 4.
\end{itemize}

The first approach uses the Cox ring of $X$,
starting from the presentation given in \cite{SX}.
Propositions \ref{prop:IXdeg4} and \ref{prop:IXdeg3} extend this
to the universal Cox ideal over the moduli space.
For any particular surface $X$, defined over a field such as $K = \mathbb{Q}(t)$,
computing the tropicalization is a task for
the software {\tt gfan} \cite{jensen}.
In the second approach, we construct del Pezzo surfaces
from fibers in the natural maps of moduli fans. Our success
along these lines completes the program started by 
Hacking {\it et al.}~\cite{HKT}
and further developed in \cite[\S 6]{RSS}.
The third approach is to build tropical del Pezzo surfaces  combinatorially
from the tropical projective plane $\mathbb{TP}^2$ by
the process of tropical modifications in the sense of Mikhalkin \cite{MikICM}.
It mirrors the classical construction 
by blowing up points in $\PP^2$.
All three approaches yield the same results. 
Section 5 presents an in-depth  study of  
the combinatorics of tropical
cubic surfaces and their trees,
 including an extension of Theorem \ref{thm:deg3}
that includes all  degenerate surfaces.

\begin{figure}[h]
\begin{center}
\vspace{-0.1in} 
\includegraphics[width=0.69\textwidth]{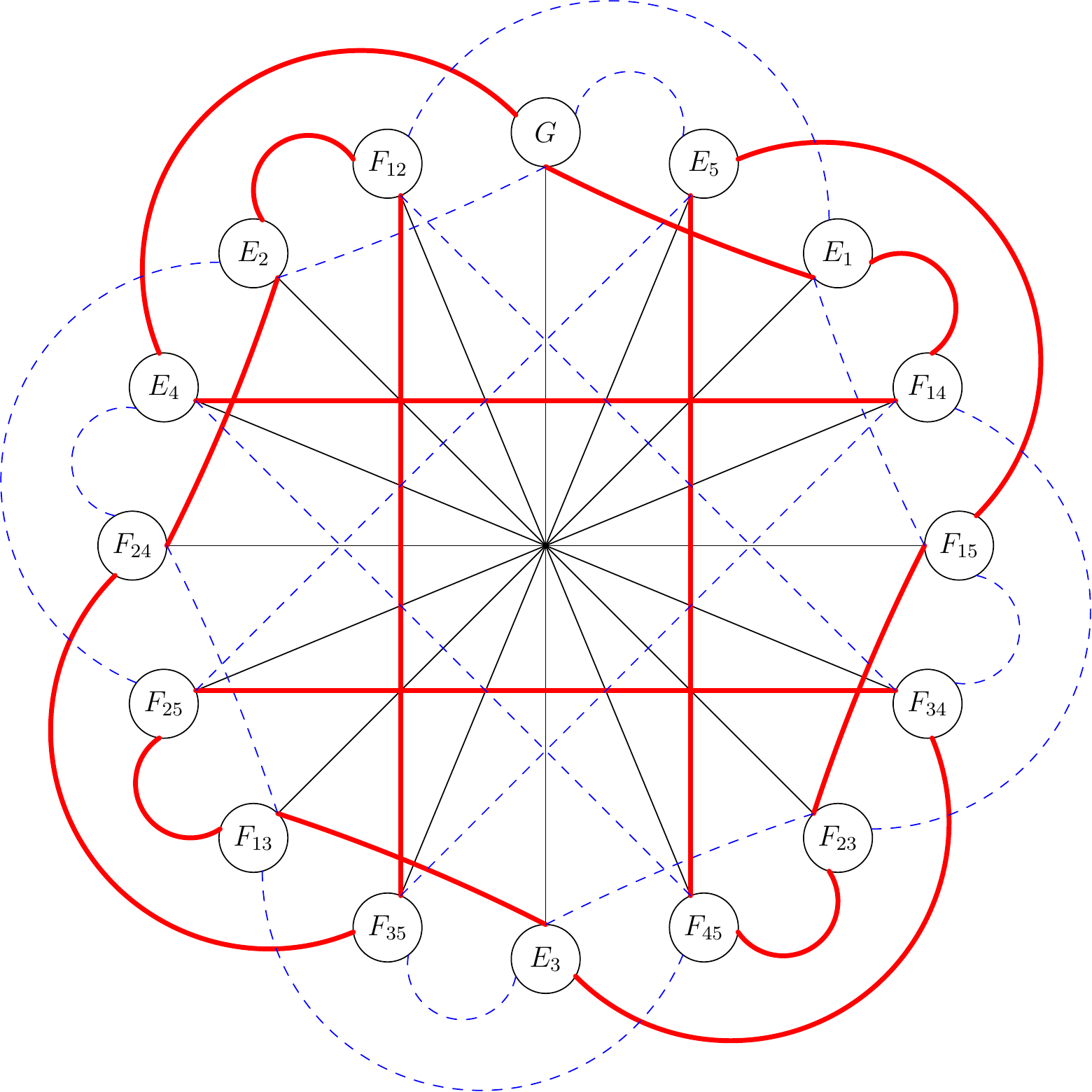}
  \vspace{-0.3in}
  \end{center}
  \caption{
    \label{fig:clebsch}
  Tropical del Pezzo surfaces of degree $4$ 
  illustrated by coloring  the Clebsch graph}
\end{figure}

\medskip

We now illustrate the rich combinatorics in our story
for a del Pezzo surface $X$ of degree~$4$.
Del Pezzo surfaces of degree $d \geq 6$ are toric surfaces, so 
they naturally  tropicalize as polygons with $12-d$ vertices
\cite[Ch.~3]{MikRau}.
On route to Theorem \ref{thm:deg3}, we prove
the following for $d=4,5$:

\begin{proposition} \label{prop:deg45}
Among tropical del Pezzo surfaces of degree $4$ and $5$,
each has a unique generic combinatorial type.
For degree $5$, this is the cone over the Petersen graph.
For degree $4$, the surface is contractible and characterized
at infinity by
$16$ trivalent metric trees, each with $5$ leaves. 
It has $9$ bounded cells, $20$ edges, $12$ vertices, $40$ cones, $32$ flaps, 
and $48$ rays.
\end{proposition}

\begin{figure}[h]
\begin{center}
\vspace{-0.15in}
\includegraphics[width=0.6\textwidth]{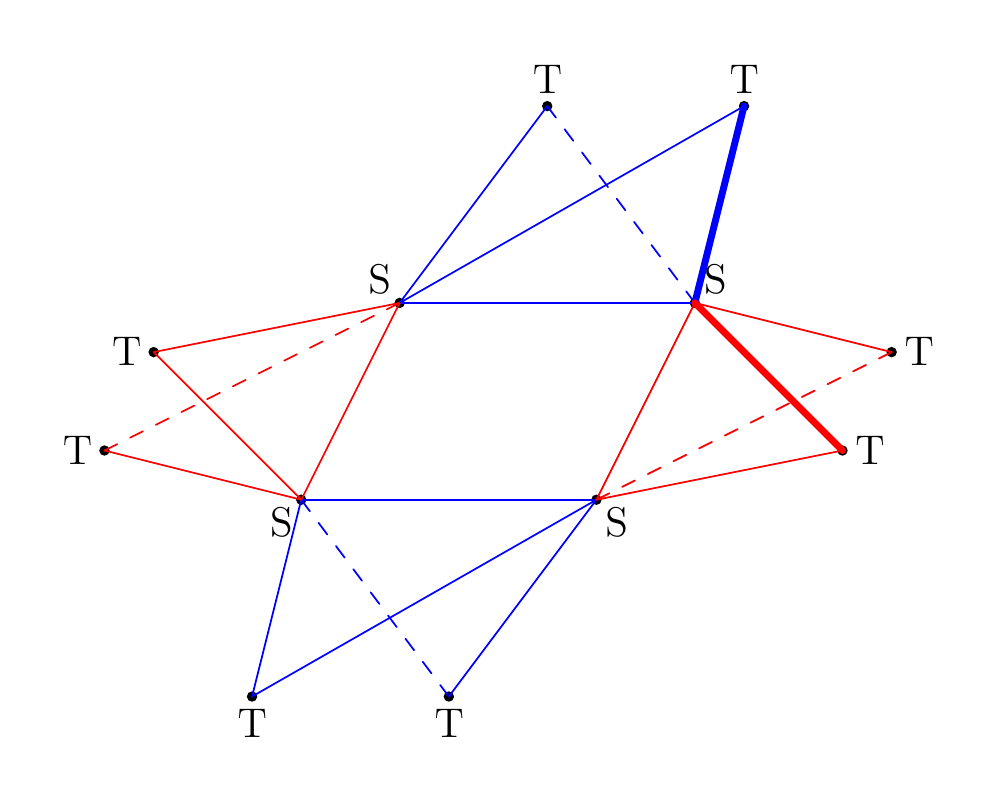}
  \vspace{-0.44in}
  \end{center}
  \caption{
   \label{fig:clebsch2}
   The bounded complex of the
  tropical del Pezzo surface in degree $4$}
\end{figure}

To understand degree $4$, we consider the
$5$-regular {\em Clebsch graph} in 
Figure \ref{fig:clebsch}.
Its $16$ nodes are the $(-1)$-curves on $X$,
labelled $E_1,\ldots,E_5,F_{12},\ldots,F_{45},G$.  Edges
represent intersecting pairs of $(-1)$-curves.
In the constant coefficient case, when
$K$ has trivial valuation, the tropicalization of $X$
is  the fan over this graph. However, over fields $K$ with non-trivial valuation,
${\rm trop}(X^0)$ is usually not a fan, but one
sees the generic type from Proposition~\ref{prop:deg45}.
Here, the Clebsch graph  deforms into a trivalent graph with $48 = 16 \cdot 3$  nodes
and $72 = 40+32$ edges, determined by the color coding in Figure~\ref{fig:clebsch}.
Each of the $16$ nodes is replaced by a trivalent tree with five leaves.
 Incoming edges of the same color
(red or blue) form a {\em cherry} (= two adjacent leaves) in that tree, while the black edge connects to
the non-cherry leaf.  

\begin{corollary}
For a del Pezzo surface $X$ of degree $4$,
the $16$ metric trees on its tropicalization ${\rm trop}(X^0)$, obtained from
the $(-1)$-curves on  $X$,
are identical up to relabeling.
\end{corollary}

\begin{proof}
Moving from one $(-1)$-curve on $X$ to another corresponds to a
Cremona transformation  of the plane $\PP^2$. Each $(-1)$-curve on 
$X$ has exactly five marked points arising from its intersections with the other
$(-1)$-curves. Moreover, the Cremona 
transformations preserve the cross ratios among the five marked
points on these $16$ $\PP^1$'s. 
From the valuations of all the various cross ratios one can read off the combinatorial 
trees along with their edge lengths, as explained in e.g.~\cite[Proposition 6.5.1]{MacStu} or
\cite[Example 5.2]{RSS}.
We then obtain
the following relabeling rules for the leaves on the $16$ trees, 
which live in the circular nodes
of Figure \ref{fig:clebsch}.

We start with the tree  $G$ whose leaves are labeled
$E_1,E_2,E_3,E_4,E_5$.  For the specific example in Figure~\ref{fig:clebsch},
this is the caterpillar tree $(\{E_1,E_4\}, E_3, \{E_2,E_5\})$.
Now, given any trivalent tree for $G$, 
the tree  $F_{ij}$
is obtained by relabeling the five leaves as follows:
\begin{equation}
\label{eq:relabel1}
 E_i \mapsto E_j,   \quad E_j \mapsto E_i, \quad 
 E_k \mapsto F_{lm}, \quad E_l \mapsto F_{km}, \quad E_m \mapsto F_{kl}.
 \end{equation}
Here $\{k,l,m\} = \{1,2,3,4,5\} \backslash \{i,j\}$.
The tree $E_i$ is obtained from the tree $G$ by relabelling
\begin{equation}
\label{eq:relabel2}
 E_i \mapsto G \qquad \hbox{and}  \qquad E_j \mapsto F_{ij} \quad \hbox{where $j \neq i$}. 
 \end{equation}
This explains the color coding of the graph in Figure \ref{fig:clebsch}.
\end{proof}




 The bounded complex of ${\rm trop}(X^0)$
is shown in Figure \ref{fig:clebsch2}. It consists of a central
rectangle, with two triangles attached to each of its four edges.
There are $12$ vertices, four vertices of the rectangle,
labeled  {\bf S},
and eight pendant vertices, labeled {\bf T}.
To  these $12$ vertices and $20$ edges, we attach 
the flaps and cones, according to the
deformed Clebsch graph structure.
The link of each ${\bf S}$ vertex in the surface ${\rm trop}(X^0)$ is the
Petersen graph (Figure \ref{fig:petersen}), while the link of each
{\bf T} vertex is the bipartite graph~$K_{3,3}$.
 The bounded complex has $16$ chains ${\bf TST}$
   consisting of two edges with different colors.
These are attached by flaps to the bounded parts of
the $16$ trees. The Clebsch graph (Figure \ref{fig:clebsch})
can be recovered from  Figure \ref{fig:clebsch2} as follows:
its nodes are the {\bf TST} chains, and
two chains connect if they share precisely one vertex.
Out at infinity, {\bf T} vertices attach along cherries,
while {\bf S} vertices attach along non-cherry leaves.
Each such attachment between two of the $16$ trees
links to the bounded complex by a cone.

\section{Cox Ideals}

We study del Pezzo surfaces over $K$
of degrees $5, 4$ and $3$.
Such surfaces $X$ are
obtained from $\PP^2$ by blowing up
$4, 5$ or $6$ points in general position, and 
we obtain moduli by varying these points. From an algebraic perspective,
it is convenient to represent $X$ by its {\em Cox ring}
\begin{equation}
\label{eq:coxring}
 {\rm Cox}(X)\,\, \,\,= \,\, \bigoplus_{\mathcal{L} \in {\rm Pic}(X)} H^0(X, \mathcal{L}). 
 \end{equation}
The Cox ring of a del Pezzo surface $X$ was first studied by Batyrev and Popov \cite{BatPop}. 
We shall express this ring explicitly as a quotient of a polynomial ring
over the ground field $K$:
$$ {\rm Cox}(X)  \,\,\, = \,\,\,
K \bigl[\, x_C \,:\, C \,\,\hbox{is a $(-1)$-curve on $X$} \bigr]  
\,\,\, \hbox{modulo
an ideal $I_X$ generated by quadrics}.$$ 
The number of variables $x_C$ in our three 
polynomial rings is $10, 16$ and $27$ respectively.
The ideal $I_X$ is the {\em Cox ideal} of the surface $X$. It was 
conjectured already in \cite{BatPop} that the ideal $I_X$ is generated by  quadrics.
This conjecture was proved in several papers, including \cite{STV, SX}.

The Cox ring encodes all maps from $X$ to a projective space.
Such a map is given by the  $\mathbb{N}$-graded subring
$\,{\rm Cox}(X)_{ [\mathcal{L}]} =  \bigoplus_{m \geq 0} H^0(X, m\mathcal{L})\,$
for a fixed line bundle $\mathcal{L} \in {\rm Pic}(X)$.
The image of the map $X \rightarrow {\rm Proj}({\rm Cox}(X)_{[\mathcal{L}]})$ 
is contained in the projective space $\PP^{N}$,
where $\,N = {\rm dim} (H^0(X, \mathcal{L}))-1$, provided
${\rm Cox}(X)_{ [\mathcal{L}]}$ is generated in degree $1$.
This applies to both the anticanonical map
and to the blow-down map to~$\PP^2$.

In what follows, we  give explicit generators for all relevant Cox ideals $I_X$.
Some of this is new and of independent interest.
The tropicalization of $X^0$ we seek is defined from the
ideal $I_X$. So, in principle, we can compute
${\rm trop}(X^0)$ from $I_X$ using the software {\tt gfan} \cite{jensen}.
Recall that $X^0$ denotes the very affine surface
obtained from $X$ by removing all $(-1)$-curves.

\medskip

\begin{figure}[h]
\begin{center}
\vspace{-0.15in} 
\includegraphics[width=0.3\textwidth]{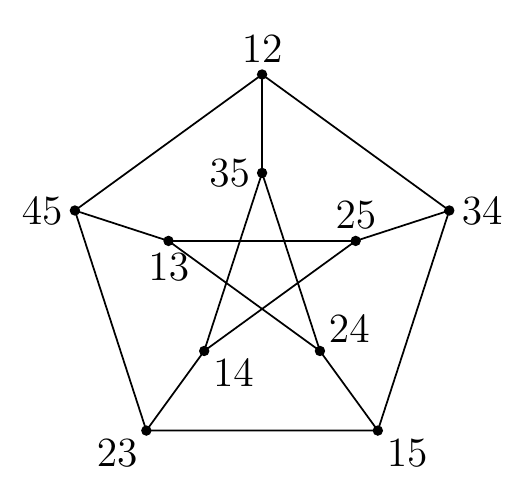}
  \vspace{-0.3in}
  \end{center}
  \caption{The tropical del Pezzo surface
  ${\rm trop}(M_{0,5}) $ is the cone over the Petersen graph.
  \label{fig:petersen}}
\end{figure}

\noindent {\bf Del Pezzo Surfaces of Degree $5$} \\
Consider four general points in $\PP^2$.
This configuration is projectively unique, so there are no moduli.
The surface $X$ is the moduli space $\overline{M}_{0,5}$
of rational stable curves with five marked points, see for example \cite{keel}.
The very affine variety $X^0$ is simply $M_{0,5}$, the moduli space of rational 
curves with five distinct marked points. It is the complement of a hyperplane arrangement 
whose underlying matroid is the graphical matroid of the complete graph $K_4$. 
The Cox ideal is the {\em Pl\"ucker ideal}
of relations among $2 \times 2$-minors of a $2 \times 5$-matrix:
$$ \begin{matrix}                                                                 
I_X &  = &  \langle\, p_{12} p_{34} - p_{13} p_{24} + p_{14} p_{23} ,            
&  \,p_{12} p_{35} - p_{13} p_{25} + p_{15} p_{23} , \\                          
& & \,\,p_{12} p_{45} - p_{14} p_{25} + p_{15} p_{24} ,                          
&  \,p_{13} p_{45} - p_{14} p_{35} + p_{15} p_{34} ,                        
&  \, p_{23} p_{45} - p_{24} p_{35} + p_{25} p_{34}\, \rangle  .          
\end{matrix}                                                                      
$$
The affine variety of  $I_X\,$ in $\bar{K}^{10}$
is the {\em universal torsor} of $X$, now regarded over
the algebraic closure $\bar{K}$ of the given valued field $K$.
From the perspective of blowing up $\PP^2$ at $4$ points,
the ten variables (representing the ten $(-1)$-curves) fall in two groups:
the four exceptional fibers, and the six lines spanned by pairs of points.
For example, we may label the fibers by
 $$  E_1 = p_{15}, \,\,
E_2 = p_{25}, \,\,
E_3 = p_{35}, \,\,
E_4 = p_{45} , $$
and the six lines by
$$ 
F_{12} = p_{34}, \,\,
F_{13} = p_{24}, \,\,
F_{14} = p_{23}, \,\,
F_{23} = p_{14}, \,\,
F_{24} = p_{13}, \,\,
F_{34} = p_{12}.$$


The Cox ideal $I_X$ is homogeneous with
respect to the natural grading 
 by the Picard group ${\rm Pic}(X) = \Z^5$.
 In Pl\"ucker coordinates, this grading is 
 given by setting ${\rm deg}(p_{ij}) 
  = e_i + e_j$, where $e_i$ 
 represents the $i^{th}$ standard basis vector in $\Z^5 = {\rm Pic}(X)$. 
 This translates into an action of the torus
 $\,(\bar{K}^*)^5 = {\rm Pic}(X) \otimes_\Z \bar{K}^*\,$
on the universal torsor in $\bar{K}^{10}$.
We remove the ten coordinate hyperplanes
in $\bar{K}^{10}$, and we take the quotient  
 modulo $(\bar{K}^*)^5$. The result is precisely the
 very affine del Pezzo surface we seek to tropicalize:
\begin{equation}
\label{embeddeg5} X^0 \,= \,M_{0,5} \,\,\, \subset \,\,\, (\bar{K}^*)^{10}/(\bar{K}^*)^5. 
\end{equation}
The $2$-dimensional balanced 
fan ${\rm trop}(X^0)$ is the Bergman fan of the graphical matroid 
of $K_4$. It is known from \cite{ArdKli}
 that this is the cone over the Petersen graph.
This is also  easy to check directly with {\tt gfan} on $I_X$. 
This fan is also
the moduli space of $5$-marked rational tropical curves,
that is, $5$-leaf trees with lengths on the two bounded edges 
(cf.~\cite[\S 4.3]{MacStu}).

\bigskip

\noindent {\bf Del Pezzo Surfaces of Degree $4$} \\
Consider now five general points in $\PP^2$. There are  two degrees of freedom.
The moduli space is  our previous del Pezzo surface  $M_{0,5}$.
 Indeed, fixing five points in $\PP^2$
corresponds to fixing a point $(p_{12},\ldots,p_{45})$ in $M_{0,5} $, using
Cox-Pl\"ucker coordinates as in (\ref{embeddeg5}). 
Explicitly, if we write the five points as a $3 \times 5$-matrix
then the $p_{ij}$ are the Pl\"ucker coordinates of its kernel.
Replacing $K$ with the previous Cox ring, we may
 consider the {\em  universal del Pezzo surface} $\mathcal{Y}$.
The {\em universal Cox ring} is the quotient of a polynomial
ring in $26=10+16$ variables:
\begin{equation}
\label{eq:cox16} K[\mathcal{Y}] \, = \,{\rm Cox}(M_{0,5})[E_1,E_2,E_3,E_4,E_5,F_{12}, 
F_{13},\ldots, F_{45}, G]/I_\mathcal{Y}.
\end{equation}
As before, $E_i$ represents the exceptional divisor over point $i$,
and $F_{ij}$ represents the line spanned by points $i$ and $j$.
The variable $G$ represents the conic spanned by the five points.

\begin{proposition} \label{prop:IXdeg4}
Up to saturation with respect to the product of the $26$ variables,
the universal Cox ideal $I_\mathcal{Y}$
for degree $4$ del Pezzo surfaces
is generated by the following $45$ trinomials:
$$
\begin{matrix}
\hbox{Base Group} & & 
p_{12}  p_{34}{-} p_{13}  p_{24}{+}p_{14}  p_{23} &
p_{12}  p_{35}{-}p_{13}  p_{25}{+}p_{15}  p_{23} &
 p_{12}  p_{45}{-}p_{14}  p_{25}+p_{15}  p_{24}, \\ & & 
 p_{13}  p_{45}{-}p_{14}  p_{35}+p_{15}  p_{34} &
p_{23}  p_{45}{-}p_{24}  p_{35}{+}p_{25}  p_{34} & 
 \end{matrix} 
 $$ \vspace{-0.27cm} $$ 
\begin{matrix}
\hbox{Group 1} & & 
F_{23}  F_{45} {-} F_{24}  F_{35} {+} F_{25}  F_{34} &
p_{23}  p_{45}  F_{24}  F_{35} {-} p_{24}  p_{35}  F_{23}  F_{45}{-} G  E_1 \\ & & 
p_{23}  p_{45}  F_{25}  F_{34} {-} p_{25}  p_{34}  F_{23}  F_{45} {-} G  E_1 & 
p_{24}  p_{35}  F_{25}  F_{34} {-} p_{25}  p_{34}  F_{24}  F_{35} {-} G  E_1 
\end{matrix}
$$ \vspace{-0.25cm}  $$
\begin{matrix}
\hbox{Group 2} & & 
F_{13}  F_{45} {-} F_{14}  F_{35} {+} F_{15}  F_{34} &
p_{13}  p_{45}  F_{14}  F_{35} {-} p_{14}  p_{35}  F_{13}  F_{45} {-} G  E_2 \\ & & 
p_{13}  p_{45}  F_{15}  F_{34} {-} p_{15}  p_{34}  F_{13}  F_{45} {-} G  E_2 &
p_{14}  p_{35}  F_{15}  F_{34} {-} p_{15}  p_{34}  F_{14}  F_{35} {-} G  E_2
\end{matrix}
$$ \vspace{-0.25cm}  $$
\begin{matrix}
\hbox{Group 3} & & 
 F_{12}  F_{45} {-} F_{14}  F_{25} {+} F_{15}  F_{24} &
 p_{12}  p_{45}  F_{14}  F_{25} {-} p_{14}  p_{25}  F_{12}  F_{45} {-} G  E_3,\\ & & 
p_{12}  p_{45}  F_{15}  F_{24} {-} p_{15}  p_{24}  F_{12}  F_{45} {-} G  E_3 &
p_{14}  p_{25}  F_{15}  F_{24} {-} p_{15}  p_{24}  F_{14}  F_{25} {-} G  E_3
\end{matrix}
$$ \vspace{-0.25cm}  $$
 \begin{matrix}
\hbox{Group 4} & & 
F_{12}  F_{35} {-} F_{13}  F_{25} {+} F_{15}  F_{23} &
p_{12}  p_{35}  F_{13}  F_{25} {-} p_{13}  p_{25}  F_{12}  F_{35} {-} G  E_4 \\ & & 
p_{12}  p_{35}  F_{15}  F_{23} {-} p_{15}  p_{23}  F_{12}  F_{35} {-} G  E_4 &
p_{13}  p_{25}  F_{15}  F_{23} {-} p_{15}  p_{23}  F_{13}  F_{25} {-} G  E_4
\end{matrix}
$$ \vspace{-0.25cm}  $$
\begin{matrix}
\hbox{Group 5} & & 
F_{12} F_{34} {-} F_{13}  F_{24} {+} F_{14}  F_{23} &
p_{12}  p_{34}  F_{13}  F_{24} {-} p_{13}  p_{24}  F_{12}  F_{34} {-} G  E_5 \\ & & 
p_{12}  p_{34}  F_{14}  F_{23} {-} p_{14}  p_{23} F_{12}  F_{34} {-} G  E_5 & 
p_{13}  p_{24}  F_{14}  F_{23} {-} p_{14}  p_{23}  F_{13}  F_{24} {-} G  E_5
\end{matrix}
$$ \vspace{-0.25cm}  $$
\begin{matrix}
\hbox{Group 1'} & \quad & 
p_{25}  F_{12}  E_2 {-} p_{35}  F_{13}  E_3 {+} p_{45}  F_{14}  E_4 & &
p_{24}  F_{12}  E_2 {-} p_{34}  F_{13}  E_3 {+} p_{45}  F_{15}  E_5 \\ & & 
p_{23}  F_{12}  E_2 {-} p_{34}  F_{14}  E_4 {+} p_{35}  F_{15}  E_5 & &
p_{23}  F_{13}  E_3 {-} p_{24}  F_{14}  E_4 {+} p_{25}  F_{15} E_5
\end{matrix}
$$ \vspace{-0.25cm}  $$
\begin{matrix}
\hbox{Group 2'} & \quad & 
p_{15}  F_{12} E_1 {-} p_{35}  F_{23}  E_3 {+} p_{45}  F_{24}  E_4 & &
p_{14}  F_{12}  E_1 {-} p_{34} F_{23}  E_3 {+} p_{45}  F_{25}  E_5 \\ & & 
p_{13}  F_{12}  E_1 {-} p_{34}  F_{24} E_4 {+} p_{35}  F_{25}  E_5 & &
p_{13}  F_{23}  E_3 {-} p_{14}  F_{24} E_4 {+} p_{15}  F_{25}  E_5
\end{matrix}
$$ \vspace{-0.25cm}  $$
\begin{matrix}
\hbox{Group 3'} & \quad & 
p_{15}  F_{13}  E_1 {-} p_{25}  F_{23}  E_2 {+} p_{45}  F_{34}  E_4 & &
p_{14}  F_{13}  E_1 {-} p_{24}  F_{23}  E_2 {+} p_{45}  F_{35}  E_5 \\ & & 
 p_{12}  F_{13}  E_1 {-} p_{24}  F_{34}  E_4 {+} p_{25}  F_{35}  E_5 & &
 p_{12}  F_{23}  E_2 {-} p_{14}  F_{34}  E_4 {+} p_{15}  F_{35}  E_5
 \end{matrix}
$$ \vspace{-0.25cm}  $$
\begin{matrix}
\hbox{Group 4'} & \quad & 
p_{15}  F_{14}  E_1 {-} p_{25}  F_{24}  E_2 {+} p_{35}  F_{34}  E_3 & &
p_{13}  F_{14}  E_1 {-} p_{23}  F_{24} E_2 {+} p_{35}  F_{45}  E_5 \\ & & 
p_{12}  F_{14}  E_1 {-} p_{23}  F_{34}  E_3 {+} p_{25}  F_{45}  E_5 & &
p_{12}  F_{24}  E_2 {-} p_{13}  F_{34}  E_3 {+} p_{15}  F_{45}  E_5
\end{matrix}
$$ \vspace{-0.25cm}  $$
\begin{matrix}
\hbox{Group 5'} & \quad & 
p_{14}  F_{15}  E_1 {-} p_{24}  F_{25} E_2 {+} p_{34}  F_{35}  E_3 & &
p_{13}  F_{15}  E_1 {-} p_{23}  F_{25}  E_2 {+} p_{34}  F_{45} E_4 \\  & & 
p_{12}  F_{15}  E_1 {-} p_{23}  F_{35}  E_3 {+} p_{24}  F_{45}  E_4 & &
p_{12}  F_{25}  E_2 {-} p_{13}  F_{35}  E_3 {+} p_{14}  F_{45} E_4
\end{matrix}
$$
\end{proposition}

Proposition  \ref{prop:IXdeg4} will be derived later in this section.
For now, let us discuss the structure and symmetry of the
generators of $I_\mathcal{Y}$.
We consider the $5$-dimensional {\em demicube}, here denoted ${\rm D}_5$. This is
the convex hull of the following $16$ points in the hyperplane $ \{a_0 = 0 \} \subset \RR^6$:
\begin{equation}
\label{eq:demicube}
 \bigl\{ \,(1,a_1,a_2,a_3,a_4,a_5) \in \{0,1\}^6 \,: \,\,a_1+a_2+a_3+a_4+a_5 \,\,\hbox{is odd} \,\bigr\}. 
 \end{equation}
The group of symmetries of ${\rm D}_5$ is the Weyl group $W({\rm D}_5)$.
It acts transitively on (\ref{eq:demicube}). 
 There is a natural bijection between the $16$ variables in the Cox ring
 and the vertices of ${\rm D}_5$:
 \begin{equation}
 \label{eq:D5grading}
 \begin{matrix}
 E_1 \leftrightarrow (1,1,0,0,0,0), \,
 E_2 \leftrightarrow (1,0,1,0,0,0),
\, \ldots \,,\,
 E_5 \leftrightarrow (1,0,0,0,0,1),   \\
 F_{12} \leftrightarrow (1,0,0,1,1,1) , \,
  F_{13} \leftrightarrow (1,0,1,0,1,1) , \,
\, \ldots\, , \,
 F_{45} \leftrightarrow (1,1,1,1,0,0), \\
   G \leftrightarrow (1,1,1,1,1,1) . 
 \end{matrix}
 \end{equation}
 This bijection defines the grading 
 via the Picard group $\Z^6$.
 We regard the $p_{ij}$ as scalars, so they have degree $0$.
Generators of $I_{\mathcal{Y}}$ that are listed in the same group have the 
same $\Z^6$ degrees. 
The action of $W({\rm D}_5)$ on the demicube ${\rm D}_5$ gives the 
action on the $16$ variables.

Consider now a particular smooth 
del Pezzo surface $X$ of degree $4$ over the field $K$,
so the $p_{ij}$ are scalars in $K$ that satisfy
the Pl\"ucker relations in the Base Group.
The universal Cox ideal $I_\mathcal{Y}$ specializes to
 the Cox ideal $I_X$ for the particular surface $X$. That Cox ideal is
minimally generated by $20$ quadrics, two per group.
The surface $X^0$  is the zero set of the ideal $I_X$
inside $(\bar{K}^*)^{16}/(\bar{K}^*)^6$. The torus action is obtained from
(\ref{eq:D5grading}), in analogy to  (\ref{embeddeg5}).

\begin{proof}[Proof of Proposition~\ref{prop:deg45}]
We  computed ${\rm trop}(X^0)$ by
applying {\tt gfan} \cite{jensen}
 to the ideal $I_X$.
If $K = \mathbb{Q}$ with the trivial valuation then the output
is the cone over the {\em Clebsch graph}  in Figure~\ref{fig:clebsch}.
This $5$-regular graph records which pairs of $(-1)$-curves intersect on $X$.
This also works over a field $K$
with  non-trivial valuation. The software {\tt gfan} uses
$K = \mathbb{Q}(t)$. If the vector
 $(p_{12},\ldots,p_{45})$ tropicalizes into the
 interior of an edge in the Petersen graph
then $ {\rm trop}(X^0)$ is the
tropical surface described in Proposition~\ref{prop:deg45}.
Each node in Figure \ref{fig:clebsch}
is replaced by a trivalent tree on $5$ nodes
according to the color coding explained in Section~\ref{sec:intro}.
The surface ${\rm trop}(X^0)$ can also be determined by tropical modifications,
as in Section \ref{sec:mod}.
\end{proof}

The same tropicalization method works  for 
the universal family $\mathcal{Y}^0$. 
Its ideal $I_{\mathcal{Y}}$ is given by the
$45$ polynomials in $26$ variables  listed above,
and $\mathcal{Y}^0$ is the
zero set of  $I_\mathcal{Y}$  in the 
$15$-dimensional torus
$\,(\bar{K}^*)^{10}/(\bar{K}^*)^5 \times (\bar{K}^*)^{16}/(\bar{K}^*)^6$.
The tropical universal del Pezzo surface ${\rm trop}(\mathcal{Y}^0)$ 
is a $4$-dimensional fan in $\R^{26}/\R^{11}$.
We compute it by
applying {\tt gfan} to the {\em universal Cox ideal}
$I_{\mathcal{Y}}$.
 The Gr\"obner fan structure on
 ${\rm trop}(I_\mathcal{Y})$ has  f-vector $(76, 630, 1620, 1215)$.
It  is isomorphic   to the {\em Naruki fan} 
described in \cite[Table 5]{RSS} and discussed further in Section~3.

\medskip

\noindent {\bf Del Pezzo Surfaces of Degree $3$ (Cubic Surfaces)} \\
There exists a cuspidal cubic through any six points in $\PP^2$.
See e.g.~\cite[(4.4)]{RSSS} and \cite[(6.1)]{RSS}.
Hence any configuration of six points in $\PP^2$ can be represented by the columns
of a matrix 
$$                                                                                
D \,\,\,= \,\,\,                                                                  
\begin{pmatrix} 1 & 1 & 1 & 1 & 1 & 1 \\ d_1 & d_2 & d_3 & d_4 & d_5 & d_6  \\\   
 d_1^3 & d_2^3 & d_3^3 & d_4^3 & d_5^3 & d_6^3 \end{pmatrix}   .                   
$$
The maximal minors of the matrix $D$ factor into linear forms,
\begin{equation}
\label{eq:degg4}
[ijk] \quad = \quad (d_i-d_j)(d_i-d_k)(d_j-d_k) (d_i+d_j+d_k),
\end{equation}
and so does the condition for the six points to lie on a conic:
\begin{equation}
\label{eq:degg16}
\begin{matrix} [{\rm conic}] \, &=\!\!& \!\!\!\!
[134][156][235][246] -  [135][146][234][256] \qquad \\ \,\,& = \!\!&
\qquad (d_1+d_2+d_3+d_4+d_5+d_6) \cdot \prod_{1 \leq i < j \leq 6}  (d_i-d_j).
\end{matrix}
\end{equation}

The linear factors in these expressions form the root system of type $\mathrm{E}_6$.
This corresponds to an arrangement of $36$ hyperplanes in $\PP^5$.
Similarly, the arrangement of type
$\mathrm{E}_7$ consists of $63$ hyperplanes in $\PP^6$, as in
\cite[(4.4)]{RSSS}.
To be precise, for $m=6,7$, the roots for $\mathrm{E}_m$ are
\begin{equation}
\label{eq:roots}
 \begin{matrix}
d_i-d_j &  \hbox{for} &  1\leq{}i<j\leq{}m, \\
d_i+d_j+d_k &  \hbox{for} & 1\leq{}i<j<k\leq{}m, \\
d_{i_1}+d_{i_2}+\dotsb{}+d_{i_6} & \hbox{for} & 1\leq{}i_1<i_2<\dotsb{}<i_6\leq{}m.
 \end{matrix} 
 \end{equation}
 Linear dependencies among these  linear forms specify a matroid of rank~$m$,
 also denoted ${\rm E}_m$.
 
The moduli space of marked cubic surfaces is 
the $4$-dimensional {\em Yoshida variety} $\,\mathcal{Y} $ defined in  \cite[\S 6]{RSS}.
 It coincides with the subvariety  $\mathcal{Y}^0$ 
of $(\bar{K}^*)^{26}/(\bar{K}^*)^{11}$  cut out
by the $45$ trinomials  in Proposition \ref{prop:IXdeg4}.
This is the embedding of $\mathcal{Y}^0$ in its intrinsic torus,
as in \cite[Theorem 6.1]{RSS}, and it differs from 
the embedding of $\mathcal{Y}^0$  into 
$(\bar{K}^*)^{40}/(\bar{K}^*)$  referred to 
 in Section~\ref{sec2} below.

 The universal family  for cubic surfaces is denoted by $\mathcal{G}^0$.
This is the open part of the {\em G\"opel variety} $\mathcal{G} \subset \PP^{134}$ constructed 
in \cite[\S 5]{RSSS}. The base of this $6$-dimensional family is the $4$-dimensional $\mathcal{Y}^0$. 
The map $\mathcal{G}^0 \rightarrow \mathcal{Y}^0$ was described in \cite{HKT}.
Thus the ring $K[\mathcal{Y}]$  in (\ref{eq:cox16}) is the natural base ring for 
the universal Cox ring for degree $3$ surfaces.

At this point it is essential to avoid confusing notation.
To  aim for a clear presentation, we use
the uniformization of $\mathcal{Y}$ by the $\mathrm{E}_6$ hyperplane arrangement.
Namely, we take
$R = \mathbb{Z}[d_1,d_2,d_3,d_4,d_5,d_6]$ instead of
 $K[\mathcal{Y}]$ as the base ring.
 We write $\mathcal{X}$ for the universal cubic surface over $R$.
The {\em universal Cox ring} is a quotient of the polynomial ring over $R$ in $27$ variables,
one for each line on the cubic surface. 
Using variable names as in \cite[\S 5]{SX}, we write
\begin{equation}
\label{eq:cox27} {\rm Cox}(\mathcal{X}) \, = \, R[E_1,E_2,\ldots,E_6,F_{12}, 
F_{13},\ldots, F_{56}, G_1 ,G_2, \ldots, G_6]/I_\mathcal{X}.
\end{equation}
This ring is graded by the Picard group $\Z^7$, similarly to  (\ref{eq:D5grading}).
The role of the $5$-dimensional demicube $\mathrm{D}_5$ is now played by the
$6$-dimensional {\em Gosset polytope} with $27$ vertices, here also 
denoted by $\mathrm{E}_6$. The symmetry group of this polytope is the Weyl group $W(\mathrm{E}_6)$.

\begin{proposition}  \label{prop:IXdeg3}
Up to saturation by the product of all $\,27$ variables and all $\,36$ roots,
 the universal Cox ideal $I_\mathcal{X}$ is generated
by  $270$ trinomials. These are
clustered by $\Z^7$-degrees into $27$ groups of $10$ generators,
one for each line on the cubic surface.
For instance, the $10$ generators of 
$I_\mathcal{X}$ that correspond to the line $G_1$
involve the $10$ lines that meet $G_1$. They are
$$    \!\!\!\!\!\!\!
\begin{matrix}
(d_3{-}d_4)(d_1{+}d_3{+}d_4)E_2F_{12} - (d_2{-}d_4)(d_1{+}d_2{+}d_4)E_3F_{13}
 + (d_2{-}d_3)(d_1{+}d_2{+}d_3)E_4F_{14} ,\\
(d_3{-}d_5)(d_1{+}d_3{+}d_5)E_2F_{12} - (d_2{-}d_5)(d_1{+}d_2{+}d_5)E_3F_{13}
+ (d_2{-}d_3)(d_1{+}d_2{+}d_3)E_5F_{15} ,\\
(d_3{-}d_6)(d_1{+}d_3{+}d_6)E_2F_{12} - (d_2{-}d_6)(d_1{+}d_2{+}d_6)E_3F_{13}
 + (d_2{-}d_3)(d_1{+}d_2{+}d_3)E_6F_{16} , \\
(d_4{-}d_5)(d_1{+}d_4{+}d_5)E_2F_{12} - (d_2{-}d_5)(d_1{+}d_2{+}d_5)E_4F_{14}
+ (d_2{-}d_4)(d_1{+}d_2{+}d_4)E_5F_{15} , \\
(d_4{-}d_6)(d_1{+}d_4{+}d_6)E_2F_{12} - (d_2{-}d_6)(d_1{+}d_2{+}d_6)E_4F_{14}
 + (d_2{-}d_4)(d_1{+}d_2{+}d_4)E_6F_{16} , \\
(d_5{-}d_6)(d_1{+}d_5{+}d_6)E_2F_{12} - (d_2{-}d_6)(d_1{+}d_2{+}d_6)E_5F_{15}
 + (d_2{-}d_5)(d_1{+}d_2{+}d_5)E_6F_{16} , \\
(d_4{-}d_5)(d_1{+}d_4{+}d_5)E_3F_{13} - (d_3{-}d_5)(d_1{+}d_3{+}d_5)E_4F_{14}
 + (d_3{-}d_4)(d_1{+}d_3{+}d_4)E_5F_{15} , \\
(d_4{-}d_6)(d_1{+}d_4{+}d_6)E_3F_{13} - (d_3{-}d_6)(d_1{+}d_3{+}d_6)E_4F_{14}
 + (d_3{-}d_4)(d_1{+}d_3{+}d_4)E_6F_{16} ,\\
(d_5{-}d_6)(d_1{+}d_5{+}d_6)E_3F_{13} - (d_3{-}d_6)(d_1{+}d_3{+}d_6)E_5F_{15}
+ (d_3{-}d_5)(d_1{+}d_3{+}d_5)E_6F_{16} ,\\
(d_5{-}d_6)(d_1{+}d_5{+}d_6)E_4F_{14} - (d_4{-}d_6)(d_1{+}d_4{+}d_6)E_5F_{15}
+ (d_4{-}d_5)(d_1{+}d_4{+}d_5)E_6F_{16}.
\end{matrix}
$$
The remaining $260$ trinomials are obtained by applying the
action of $W(\mathrm{E}_6)$. The variety
defined by $I_\mathcal{X}$ in
$\PP^5 \times (\bar{K}^*)^{27}/(\bar{K}^*)^{7}$ 
is $6$-dimensional. It is  the universal family $\mathcal{X}^0$.
\end{proposition}

\begin{proof}[Proof of Propositions \ref{prop:IXdeg4} and \ref{prop:IXdeg3}]
We consider the prime ideal in \cite[\S 6]{RSSS} that defines the embedding 
of the G\"opel variety $\mathcal{G}$ into $\PP^{134}$. By
 \cite[Theorem 6.2]{RSSS}, $\mathcal{G}$ is the ideal-theoretic
 intersection of a $35$-dimensional toric variety $\mathcal{T}$ and
 a $14$-dimensional linear space~$L$. The latter is cut out
 by a canonical set of $315 $ linear trinomials,
  indexed by the $315$ isotropic planes in $(\mathbb{F}_2)^6$.
Pulling these linear forms back to the Cox ring of $\mathcal{T}$,
we obtain $315$ quartic trinomials in $63$ variables, one for
each root of $\mathrm{E}_7$. Of these $63$ roots, precisely $27$ involve
the unknown $d_7$. We identify these with the
$(-1)$-curves on the cubic surface~via
\begin{equation}
\label{eq:dtoEFG} d_i - d_7 \,\mapsto\, E_i  , \qquad
 d_i+d_j + d_7 \,\mapsto\, F_{ij}  , \qquad
- d_j + \sum_{i=1}^7 d_i  \,\mapsto \, G_j. 
\end{equation}
Moreover, of the $315$ quartics, precisely $270$ contain a root 
involving $d_7$. Their images under the map (\ref{eq:dtoEFG})  
are the $270$ Cox relations
listed above. Our construction ensures that they generate the correct
Laurent polynomial ideal on
the torus of $\mathcal{T}$. 
This  proves Proposition~\ref{prop:IXdeg3}.

The derivation of Proposition  \ref{prop:IXdeg4} is similar, but now we use the substitution
$$ d_i - d_6 \,\mapsto\, E_i  , \qquad
 d_i+d_j + d_6 \,\mapsto\, F_{ij}  , \qquad
\sum_{i=1}^6 d_i \,\mapsto \, G. $$
We consider the $45$ quartic trinomials that do not involve $d_7$.
Of these, precisely five do not involve $d_6$ either. They translate
into the five Pl\"ucker relations for $M_{0,5}$. With this identification,
the remaining $40$ quartics translate into the ten groups
listed after Proposition~\ref{prop:IXdeg4}.
\end{proof}

\begin{remark} \rm
The relations (\ref{eq:dtoEFG}) is not unique. The non-uniqueness comes from both the symmetry of $\mathrm{E}_7$ and the choice of $\pm{}$ sign for each variable $E_i, F_{ij}, G_j$. For the rest of this paper, we fix the relations (\ref{eq:dtoEFG}). 
Other choices give the same result up to symmetry.
\end{remark}

We now fix a $K$-valued point in the base $\mathcal{Y}^0$,
by replacing the unknowns $d_i$ with scalars in $K$. 
In order for the resulting surface $X$ to be smooth,
we require $(d_1:d_2:d_3:d_4:d_5:d_6)$ to be in the complement
of the $36$ hyperplanes for $\mathrm{E}_6$.
The corresponding specialization of $I_\mathcal{X}$ is the
Cox ideal $I_X$ of $X$. Seven of the ten trinomials in each degree are redundant
over $K$.
Only three are needed to generate $I_X$. Hence,
the Cox ideal $I_X$ is minimally generated by $81$ quadrics
in the $E_i, F_{ij}$ and $G_i$. 
Its variety is the surface
$\, X^0 = V(I_X) \subset  (\bar{K}^*)^{27}/(\bar{K}^*)^{7}$.

\begin{proposition}\label{prop:invol}
Each of the marked $27$  trees on a tropical cubic surface has an  involution.
\end{proposition}

\begin{proof}
 Every line $L$ on a cubic surface
 $X$ over $K$,  with its ten marked points,  admits  a double cover to  $\PP^1$ with five markings. The preimage of one of these marked points is the pair of markings on $L$ given by two other lines forming a  tritangent  with $L$. Tropically, this gives a double cover from the $10$-leaf tree for $L$ to a $5$-leaf tree with leaf labelings given by these pairs. The desired involution on the $10$-leaf tree exchanges elements in each pair.
 \end{proof}
 
For instance, for the tree that corresponds to the line $L = G_1$, the involution equals
$$ 
E_2 \leftrightarrow F_{12}, \,\,\,
E_3 \leftrightarrow F_{13}, \,\,\,
E_4 \leftrightarrow F_{14}, \,\,\,
E_5 \leftrightarrow F_{15}, \,\,\,
E_6 \leftrightarrow F_{16}. $$
Indeed, this involution fixes the  $10$ Cox relations displayed in Proposition \ref{prop:IXdeg3}. The induced action on the trees corresponding to the tropicalization of the lines can be seen in Figures \ref{fig:treeaaaa} and \ref{fig:treeaaab}, where the involution reflects about a vertical axis.
 The corresponding $5$-leaf tree is the tropicalization of the line   in
    $$\mathrm{Proj}(K[E_2F_{12},E_3F_{13},E_4F_{14},E_5F_{15},E_6F_{16}])
\,\,\, \simeq \,\,\, \PP^4 $$ 
 that is the intersection of  the $10$ hyperplanes   
defined by the polynomials in Proposition~\ref{prop:IXdeg3}.

We aim to compute ${\rm trop}(X^0)$ by
applying {\tt gfan} to the ideal $I_X$.
This works well for $K = \mathbb{Q}$ with the trivial valuation.
Here the output is the cone over the {\em Schl\"afli graph}
which records which pairs of $(-1)$-curves intersect on $X$.
This is a $10$-regular graph with $27$ nodes.
However, for $K = \mathbb{Q}(t)$, 
our  {\tt gfan} calculations did not terminate. 
 Future implementations of tropical algorithms
 will surely succeed; see also Conjecture~\ref{conj:tropbasis}.
To get the tropical cubic surfaces,
and to prove Theorem~\ref{thm:deg3}, 
 we used the alternative method explained in Section \ref{sec2}.

\section{Sekiguchi Fan to Naruki Fan}
\label{sec2}

In the previous section we discussed the computation of  tropical
del Pezzo surfaces
directly from their Cox ideals. This worked well for degree $4$. However, 
 using  the current implementation in {\tt gfan}, this
 computation did not terminate for degree $3$.
We here discuss an alternative method that did succeed.
In particular, we present the proof of~Theorem \ref{thm:deg3}.

The successful computation uses the following commutative diagram of balanced fans:
\begin{equation}
\label{equation:fan_commutative_diagram}
\begin{CD}
\mathrm{Berg}(\mathrm{E}_7) @>>> \mathrm{trop}(\mathcal{G}^0)\\
@VVV @VVV\\
\mathrm{Berg}(\mathrm{E}_6) @>>> \mathrm{trop}(\mathcal{Y}^0)
\end{CD}
\end{equation}
This diagram was first derived by Hacking {\it et al.}~\cite{HKT},
in their study of moduli spaces of marked del Pezzo surfaces.
Combinatorial details were worked out by Ren {\it et al.}~in~\cite[\S 6]{RSS}.
The material that follows completes the program that was 
suggested at the very end of~\cite{RSS}.

The notation $\mathrm{Berg}({\rm E}_m)$ denotes the {\em Bergman fan} of the
rank $m$ matroid defined by the ($36$ resp.~$63$) linear forms listed in (\ref{eq:roots}).
Thus, $\mathrm{Berg}(\mathrm{E}_6)$ is a tropical linear space in $\T\PP^{35}$, 
and $\mathrm{Berg}(\mathrm{E}_7)$ is a tropical linear space in $\T\PP^{62}$.
Coordinates are labeled by roots.

The list (\ref{eq:roots}) fixes a choice of  injection
of root systems $\mathrm{E}_6\hookrightarrow{}\mathrm{E}_7$. This defines 
coordinate projections $\mathbb{R}^{63}\to{}\mathbb{R}^{36}$ 
and $\T\PP^{62} \dashrightarrow \T\PP^{35}$,
namely by deleting coordinates with index $7$.
This projection induces  the vertical map from $\mathrm{Berg}(\mathrm{E}_7)$ to $\mathrm{Berg}(\mathrm{E}_6)$
on the left in (\ref{equation:fan_commutative_diagram}).

On the right in (\ref{equation:fan_commutative_diagram}), we see
 the $4$-dimensional {\em Yoshida variety} $\mathcal{Y} \subset \PP^{39}$
and the  $6$-dimensional {\em G\"opel variety} $\mathcal{G} \subset \PP^{134}$. 
Explicit parametrizations and equations for these varieties were presented in
 \cite{RSSS, RSS}.  The corresponding very affine varieties
 $\mathcal{G}^0 \subset ({\bar K}^*)^{135}/{\bar K}^*$ and
 $\mathcal{Y}^0 \subset ({\bar K}^*)^{40}/{\bar K}^*$ 
  are moduli spaces of marked del Pezzo surfaces \cite{HKT}.
   Their tropicalizations $\mathrm{trop}(\mathcal{G}^0)$ and 
 $\mathrm{trop}(\mathcal{Y}^0)$ are known as the {\it Sekiguchi fan} and {\it Naruki fan}, respectively.
 The modular interpretation in \cite{HKT} ensures the existence of the vertical map
 $\mathrm{trop}(\mathcal{G}^0)\to{}\mathrm{trop}(\mathcal{Y}^0)$.
 This map is described in  \cite[Lemma 5.4]{HKT}, in~\cite[(6.5)]{RSS},
 and in the proof of Lemma~\ref{lem:bignumbers} below.

 The two horizontal maps in
(\ref{equation:fan_commutative_diagram}) are  surjective
and (classically) linear.
  The linear map $\mathrm{Berg}(\mathrm{E}_7)\to{}\mathrm{trop}(\mathcal{G}^0)$ is given
         by the  $135\times{}63$ matrix $A$
   in \cite[\S 6]{RSSS}. The corresponding toric variety is
  the object of \cite[Theorem 6.1]{RSSS}.
The map  
   $\mathrm{Berg}(\mathrm{E}_6)\to{}\mathrm{trop}(\mathcal{Y}^0)$ is given
   by the $40 \times 36$-matrix in \cite[Theorem 6.1]{RSS}.
We record  the following computational result.
It refers to the  natural simplicial fan structure on
  $\mathrm{Berg}(E_m)$ described by
  Ardila {\it et al.}~in \cite{ARW}.
  
\begin{lemma}
\label{lem:bignumbers}
The Bergman fans of $\mathrm{E}_6$ and $\mathrm{E}_7$ have dimensions $5$ and $6$.
Their f-vectors are
$$ \begin{matrix}
f_{\mathrm{Berg}(\mathrm{E}_6)} &=& (1, 750, 17679, 105930, 219240, 142560) , \qquad \qquad \\
f_{\mathrm{Berg}(\mathrm{E}_7)} &=&  (1, 6091, 315399, 3639804, 14982660, 24607800, 13721400).
\end{matrix} 
$$
The moduli fans  $\mathrm{trop}(\mathcal{Y}^0)$ and $\mathrm{trop}(\mathcal{G}^0)$ 
have dimensions $4$ and $6$. Their f-vectors are
$$ \begin{matrix}
f_{\mathrm{trop}(\mathcal{Y}^0)} &=&  (1, 76, 630, 1620, 1215), \qquad \qquad \\
f_{\mathrm{trop}(\mathcal{G}^0)} &=& \qquad  \quad (1, 1065, 27867, 229243, 767025, 1093365, 547155).
\end{matrix}
$$
\end{lemma}

\begin{proof}
The f-vector for  the Naruki fan
$\mathrm{trop}(\mathcal{Y}^0)$
appears in \cite[Table 5]{RSS}.
For the other three fans, only the numbers  of rays (namely
$750, 6091$ and $1065$) were known from \cite[\S 6]{RSS}.
The main new result in Lemma~\ref{lem:bignumbers} is the computation of all
$57273155$ cones
in $\mathrm{Berg}(\mathrm{E}_7)$.
The fans $\mathrm{Berg}(\mathrm{E}_6)$ and $\mathrm{trop}(\mathcal{G}^0)$
are subsequently 
derived from $\mathrm{Berg}(\mathrm{E}_7)$
using the maps in (\ref{equation:fan_commutative_diagram}).

We now describe how $f_{\mathrm{Berg}(\mathrm{E}_7)}$ was found.
We did not use the theory of tubings in \cite{ARW}. Instead,
we carried out a brute force computation based on \cite{FS} and \cite{Rin}. 
Recall that a point lies in the Bergman fan of a matroid if and only if the minimum is obtained twice on each circuit. We computed all circuits of the
 rank $7$ matroid on the $63$ vectors in the root system $\mathrm{E}_7$.
That matroid has precisely $100662348$ circuits. Their cardinalities range from $3$ to $8$. This 
furnishes a subroutine for deciding whether a given point lies in the Bergman fan.

Our computations were mostly done in {\tt sage} \cite{Sage} and {\tt java}. We achieved speed by
exploiting the action of the Weyl group
$W(\mathrm{E}_7)$ given by the two generators in \cite[(4.2)]{RSSS}. 
The two matrices derived from these two generators using \cite[(4.3)]{RSSS}
act on the space $\R^7$ with coordinates $d_1,d_2,\ldots,d_7$.
This gives subroutines for the action of $W(\mathrm{E}_7)$ on $\mathbb{R}^{63}$, e.g.~for 
deciding whether two given sequences of points are conjugate with respect to this action. 

Let $\mathbf{r}_1,\dotsc{},\mathbf{r}_{6091}$ denote the rays of $\mathrm{Berg}(\mathrm{E}_7)$,
as  in \cite[Table 2]{HKT} and \cite[\S 6]{RSS}.
 They form $11$ orbits under the action of $W(\mathrm{E}_7)$. For each orbit, we
 take  the representative $\mathbf{r}_i$ with smallest label. For each pair $i<j$ such that
 $\mathbf{r}_i$ is a representative, our program checks if
  $\mathbf{r}_i+\mathbf{r}_j$ lies in $\mathrm{Berg}(\mathrm{E}_7)$, using
    the precomputed list of circuits.
     If yes, then $\mathbf{r}_i$ and $\mathbf{r}_j$ span a $2$-dimensional cone in $\mathrm{Berg}(\mathrm{E}_7)$.
This process gives representatives for the $W(\mathrm{E}_7)$-orbits
of $2$-dimensional cones. The list of all cones is produced
   by applying the action of $W(\mathrm{E}_7)$ on the result. For each orbit,
we keep only the lexicographically smallest representative 
$(\mathbf{r}_i,\mathbf{r}_j)$.

Next, for each triple $i < j < k$ such that $(\mathbf{r}_i,\mathbf{r}_j)$
is a representative, 
  we check if $\mathbf{r}_i+\mathbf{r}_j+\mathbf{r}_k$ lies in $\mathrm{Berg}(\mathrm{E}_7)$. 
  If so, then $\{   \mathbf{r}_i, \mathbf{r}_j, \mathbf{r}_k \}$
     spans a $3$-dimensional cone in $\mathrm{Berg}(\mathrm{E}_7)$.
   The list of all $3$-dimensional cones can be found
   by applying the action of $W(\mathrm{E}_7)$ on the result. As before, we fix the 
   lexicographically smallest representatives. Repeating this process for dimensions
   $4,5$ and $6$,  we obtain the list of  all cones in $\mathrm{Berg}(\mathrm{E}_7)$,
   and hence the f-vector of this fan.

We now describe the procedure to derive $\mathrm{trop}(\mathcal{G}^0)$ by applying the top horizontal map $ \phi:\mathrm{Berg}(\mathrm{E}_7)\to{}\mathrm{trop}(\mathcal{G}^0)$.
Each ray $\mathbf{r}$ in $\mathrm{Berg}(\mathrm{E}_7)$ maps to either (a) $\mathbf{0}$, (b) a ray of $\mathrm{trop}(\mathcal{G}^0)$, or (c) a
positive linear combination of $2$ or $3$ rays, as listed in \cite[\S 6]{RSS}. For each ray in case (c), our program iterates through all pairs and triples of rays in $\mathrm{trop}(\mathcal{G}^0)$ and writes the image explicitly as a
positive linear combination of rays. With this data, we give a first guess of $\mathrm{trop}(\mathcal{G}^0)$ as follows: for each maximal cone
 $\sigma =\mathrm{span}(\mathbf{r}_{i_1},\dotsc{},\mathbf{r}_{i_6})$ of $\mathrm{Berg}(\mathrm{E}_7)$, we write
     $\phi(\mathbf{r}_{i_1}),\dotsc{},\phi(\mathbf{r}_{i_6})$ as linear combinations of the rays of $\mathrm{trop}(\mathcal{G}^0)$ and take $\sigma'\subset{}
     \T\PP^{134}     $ to be the cone spanned by all rays of $\mathrm{trop}(\mathcal{G}^0)$  that appear in the linear combinations. From this we get a list of 
     $6$-dimensional cones $\sigma'$. Let $\Phi{}\subset{}\T\PP^{134}$ be the union of these cones.
     
To certify that $\Phi = {\rm trop}(\mathcal{G}^0)$ we need to show (1) for each $\sigma\in{}\mathrm{Berg}(\mathrm{E}_7)$, we have $\phi{}(\sigma)\subset{}\sigma'$ for some cone $\sigma'\subset{}\Phi{}$; (2) each cone $\sigma'\subset{}\Phi{}$ is the union of some $\phi{}(\sigma)$ for $\sigma\in{}\mathrm{Berg}(\mathrm{E}_7)$; and (3) the intersection of any two cones $\sigma_1'$, $\sigma_2'$ in $\Phi{}$ is a face of both $\sigma_1'$ and $\sigma_2'$. The claim (1) follows from the procedure of constructing $\Phi{}$. For (2), one only needs to verify the cases where $\sigma'$ is one of the $9$ representatives by the action of $W(\mathrm{E}_7)$. For each of these, our program produces a list of $\phi{}(\sigma)$, and we check manually that $\sigma'$ is indeed the
union.
  For (3), one only needs to iterate through the cases where $\sigma_1'$ is a representative, and the procedure is straightforward. Therefore, our procedure shows that $\Phi$ is exactly $\mathrm{trop}(\mathcal{G}^0)$.
Then the $f$-vector is obtained from the list of all cones in  the fan $\Phi$. 

Finally, we recover the list of all cones in $\mathrm{Berg}(\mathrm{E}_6)$ and $\mathrm{trop}(\mathcal{Y}^0)$ by following the same procedure with the left vertical map and the bottom horizontal map.
\end{proof}

\begin{remark} \rm
The (reduced) Euler characteristic of the link of
 $\mathrm{Berg}(\mathrm{E}_7)$  is
the alternating sum of the entries of the f-vector
of  this Bergman fan.
We see from Lemma~\ref{lem:bignumbers} that this~is
\begin{align*}
1 - 6091 + 315399 - 3639804 +& 14982660 - 24607800 + 13721400 \\
&= \,765765 \,\,=\,\, 1 \cdot 5 \cdot 7 \cdot 9 \cdot 11 \cdot 13 \cdot 17.
\end{align*}
This is the product of all exponents of $W(\mathrm{E}_7)$, thus 
confirming the prediction in \cite[(9.2)]{RSSS}.
\end{remark}

The Naruki fan $\mathrm{trop}(\mathcal{Y}^0)$ is studied in \cite[\S 6]{RSS}. Under the action of $W(\mathrm{E}_6)$ through $\mathrm{Berg}(\mathrm{E}_6)$, it has two 
classes of rays, labelled type (a) and type (b). It also has two $W(\mathrm{E}_6)$-orbits of maximal cones: there are $135$ type (aaaa) cones, each spanned by four type (a) rays, and $1080$ type (aaab) cones, each spanned by three type (a) rays and one type (b) ray.

The   map $\mathrm{trop}(\mathcal{G}^0)\to{}\mathrm{trop}(\mathcal{Y}^0)$ 
tropicalizes the  morphism 
$\mathcal{G}^0 \to{}\mathcal{Y}^0$ between very affine $K$-varieties
of dimension $6$ and $4$. That morphism is the universal family 
of cubic surfaces.
 In order to tropicalize these surfaces, we examine the fibers of the 
 map $\mathrm{trop}(\mathcal{G}^0)\to{}\mathrm{trop}(\mathcal{Y}^0)$.
The next lemma concerns the subdivision of 
$\mathrm{trop}(\mathcal{Y}^0)$  induced by this map. By 
definition, this is the coarsest subdivision such that each cone in 
$\mathrm{trop}(\mathcal{G}^0)$ is sent to a union of cones.

\begin{lemma} \label{lem:barycentric}
The subdivision induced by the map $\mathrm{trop}(\mathcal{G}^0)\to{}\mathrm{trop}(\mathcal{Y}^0)$
  is the barycentric subdivision on type (aaaa) cones. For type (aaab) cones, each cone in the subdivision is a cone spanned by the type (b) ray and a cone in the barycentric subdivision of the  (aaa) face. 
Thus each (aaaa) cone is divided into $24$ 
cones, and each  (aaab) cone is divided into $6$ cones.
\end{lemma}

\begin{proof}
The map $\pi:\mathrm{trop}(\mathcal{G}^0)\to{}\mathrm{trop}(\mathcal{Y}^0)$ can be defined
via the commutative diagram~(\ref{equation:fan_commutative_diagram}): for $\mathbf{x}\in{}\mathrm{trop}(\mathcal{G}^0)$,
 take any point in its preimage in $\mathrm{Berg}(\mathrm{E}_7)$, then follow the left vertical map and the bottom horizontal map to get $\pi(\mathbf{x})$ in $\mathrm{trop}(\mathcal{Y}^0)$. It is well-defined
  because the kernel of the  map $\mathrm{Berg}(\mathrm{E}_7)\to{}\mathrm{trop}(\mathcal{G}^0)$ is contained in the kernel of the composition $\mathrm{Berg}(\mathrm{E}_7)\to{}\mathrm{Berg}(\mathrm{E}_6)\to{}\mathrm{trop}(\mathcal{Y}^0)$. 
 With this, we can compute the image in $\mathrm{trop}(\mathcal{Y}^0)$ of
 any cone in $\mathrm{trop}(\mathcal{G}^0)$. For each orbit of cones in $\mathrm{trop}(\mathcal{Y}^0)$, pick a representative $\sigma{}$, and examine
  all cones in $\mathrm{trop}(\mathcal{G}^0)$ that  map into $\sigma{}$.  Their
  images reveal the subdivision of $\sigma{}$.
\end{proof}

\begin{figure}[h]
\centering
\includegraphics[scale=0.93]{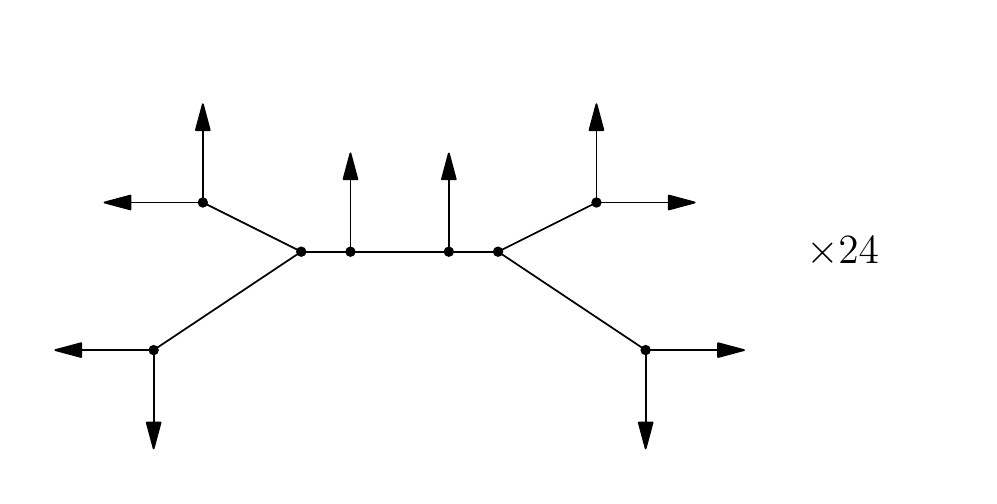}
\includegraphics[scale=0.9]{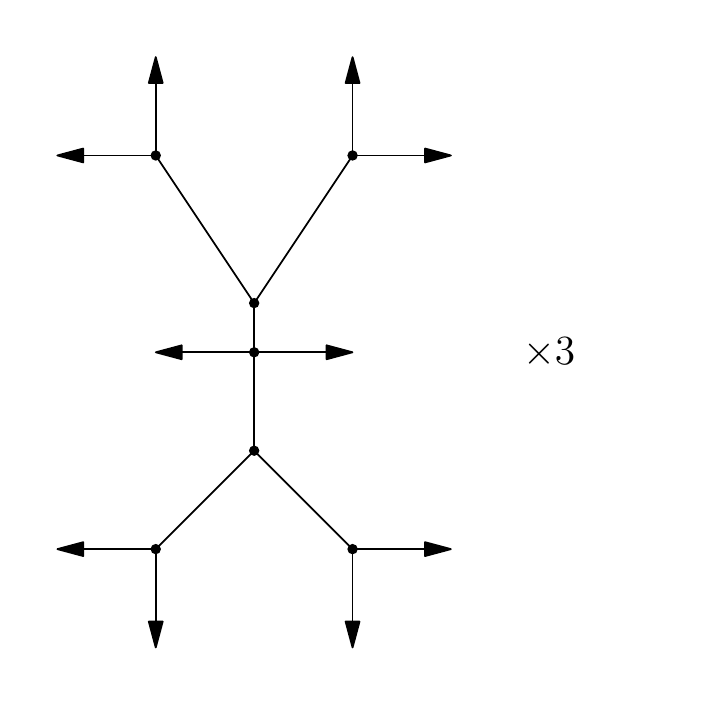}
\vspace{-0.2in}
\caption{The $27$ trees on tropical cubic surfaces of type (aaaa)
\label{fig:treeaaaa}}
\end{figure}

Lemma \ref{lem:barycentric} shows that each (aaaa) cone
of the Naruki fan ${\rm trop}(\mathcal{Y}^0)$ is divided into
$24$ subcones, and  each (aaab) cone is divided into $6$ subcones.
Thus, the total number of cones in the subdivision is $24\times{}135+6\times{}1080=9720$. 
For the base points in the interior of a cone, the fibers are contained in the same set of cones in $\mathrm{trop}(\mathcal{G}^0)$.
The fiber changes continuously as the base point changes.
Therefore, moving the base point around the interior of a cone simply 
changes the metric but not the combinatorial type of 
marked tropical cubic surface.

\begin{corollary}
The map $\,\mathrm{trop}(\mathcal{G}^0)\to{}\mathrm{trop}(\mathcal{Y}^0)$ has at most
two combinatorial types of generic fibers up to relabeling. 
\end{corollary}

\begin{proof}
We  fixed an inclusion $\mathrm{E}_6\hookrightarrow{}\mathrm{E}_7$ in (\ref{eq:roots}).
   The action of $\mathrm{Stab}_{\mathrm{E}_6}(W(\mathrm{E}_7))$ on the fans is compatible with the entire commutative diagram (\ref{equation:fan_commutative_diagram}). Hence, the fibers over two points that are conjugate under this action have the same combinatorial type.
We verify that the $9720$ cones form exactly two orbits under this action. One orbit consists of the cones in the type (aaaa) cones, and the other consists of the cones in the type (aaab) cones. Therefore, there are at most two combinatorial types, one for each orbit.
\end{proof}

\begin{figure}[h]
\centering
\includegraphics[scale=0.93]{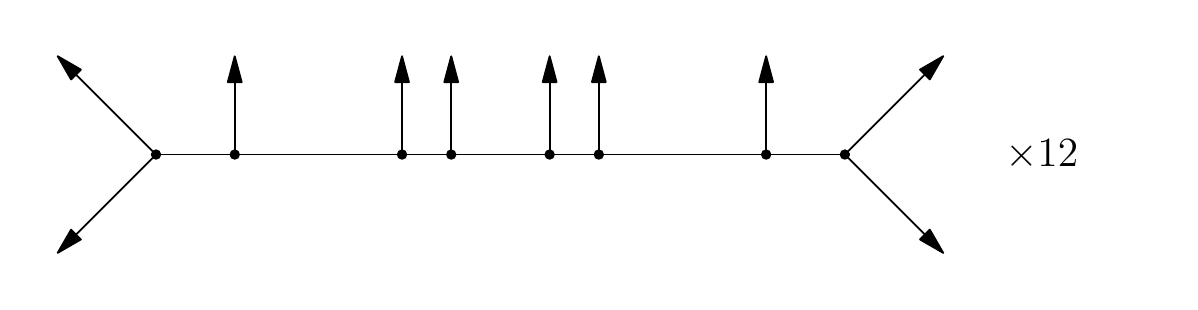}
\includegraphics[scale=0.88]{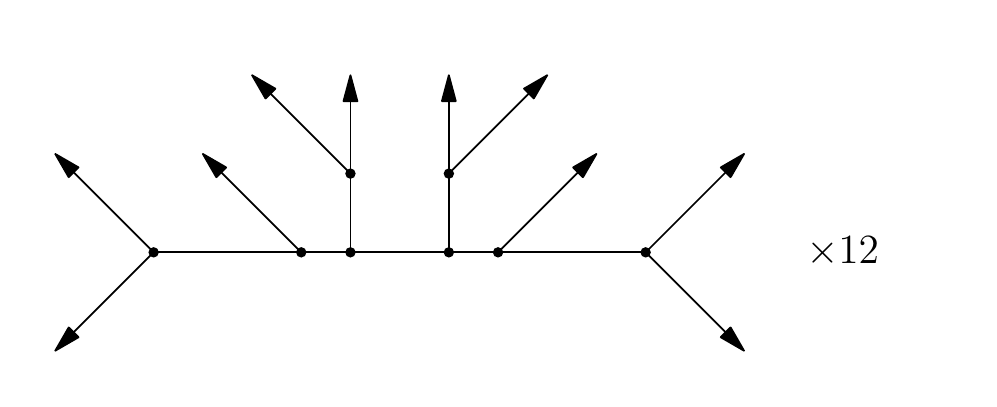}
\includegraphics[scale=0.8]{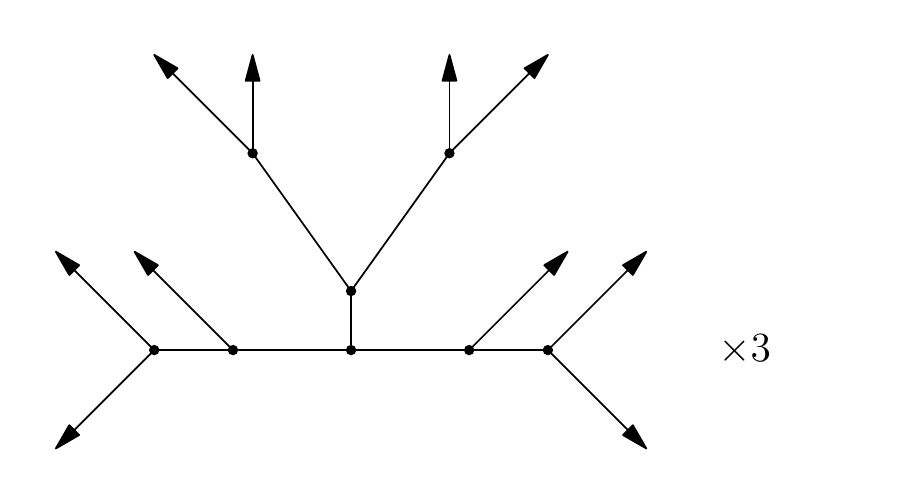}
\vspace{-0.2in}
\caption{The $27$ trees on tropical cubic surfaces of type (aaab)
\label{fig:treeaaab}}
\end{figure}

We can now derive our classification theorem for
tropical cubic surfaces.

\begin{proof}[Proof of Theorem \ref{thm:deg3}]
We compute the two types of fibers of 
$\,\pi:\mathrm{trop}(\mathcal{G}^0)\to{}\mathrm{trop}(\mathcal{Y}^0)$.
In what follows we explain this for a cone $\sigma$ of type (aaaa).
The computation for type (aaab) is similar.
Let $\mathbf{r}_1,\mathbf{r}_2,\mathbf{r}_3,\mathbf{r}_4$ denote the
rays that generate $\sigma$. We fix the vector
 $\mathbf{x} = \mathbf{r}_1+2\mathbf{r}_2+3\mathbf{r}_3+4\mathbf{r}_4$
 that    lies in the interior of a cone in the barycentric subdivision.  
   
The fiber $\pi^{-1}(\mathbf{x})$ is found by an explicit
computation. First we determine the directions of the rays.
 They arise from
rays of $\mathrm{trop}(\mathcal{G}^0)$ 
that are mapped to zero by $\pi$.
There are  $27$ such ray directions in
$\pi^{-1}(\mathbf{x})$. These are exactly the image of the $27$ type $\mathrm{A}_1$ rays in $\mathrm{Berg}(\mathrm{E}_7)$ that correspond to the roots in $\mathrm{E}_7\backslash{}\mathrm{E}_6$. We label them by $E_i,F_{ij},G_j$ as in (\ref{eq:dtoEFG}).
Next, we compute the vertices of $\pi^{-1}(\mathbf{x})$.
They are contained in $4$-dimensional cones 
$\sigma' = {\rm pos}\{\mathbf{R}_1,\mathbf{R}_2,\mathbf{R}_3,\mathbf{R}_4\}$
with $\mathbf{x} \in \pi(\sigma')$. The coordinates of each vertex
in  $\mathbb{TP}^{134}$ is computed by solving
 $y_1\pi{}(\mathbf{R}_1)+y_2\pi{}(\mathbf{R}_2)+y_3\pi{}(\mathbf{R}_3)+y_4\pi{}(\mathbf{R}_4) = \mathbf{x}$ for $y_1,y_2,y_3,y_4$. 
 
 The part of the fiber contained in each cone in $\mathrm{trop}(\mathcal{G}^0)$ is spanned by the vertices and the $E_i,F_{ij},G_j$ rays it contains. 
Iterating through the list of cones and looking at this data, we get a list that characterizes the polyhedral complex $\pi^{-1}( \mathbf{x})$.
 In particular, that list verifies that $\pi^{-1}( \mathbf{x})$ is $2$-dimensional and
 has the promised f-vector.
  For each of the $27$ ray directions $E_i,F_{ij}, G_j$, there is a tree  at infinity.
It is the link of the corresponding point at infinity 
$\pi^{-1}( \mathbf{x}) \subset \mathbb{TP}^{134}$.
The combinatorial types of these
$27$ trees  are shown in Figure \ref{fig:treeaaaa}.
 The metric on each tree can be computed as follows:
  the length of a bounded edge equals the lattice distance
between the two vertices in the corresponding flap.

The  surface $\pi^{-1}({\bf x})$ is
 homotopy equivalent to its bounded complex.
 We check directly that the bounded complex is contractible.
This can also be inferred from
Theorem~\ref{thm:modify}.
  \end{proof}

\begin{remark} \label{rem:paraxxx} \rm
We  may replace $\mathbf{x} = \mathbf{r}_1+2\mathbf{r}_2+3\mathbf{r}_3+4\mathbf{r}_4$ with a generic point $\mathbf{x} = x_1\mathbf{r}_1+x_2\mathbf{r}_2+x_3\mathbf{r}_3+x_4\mathbf{r}_4$, where $x_1{<}x_2{<}x_3{<}x_4$. This lies in the same cone in the barycentric subdivision, so 
the combinatorics of $\pi^{-1}( \mathbf{x})$
remains the same. Repeating the last step
 over the field $\mathbb{Q}(x_1,x_2,x_3,x_4)$ instead of $\mathbb{Q}$, we write the length of each 
 bounded edge in the $27$ trees in terms of the parameters. Each length
 either equals $x_1,x_2,x_3,x_4$ or is $x_i - x_j$ for some $i,j$.
\end{remark}

\section{Tropical modifications}\label{sec:mod}

In Section 2 we computed tropical varieties
from polynomial ideals, along the lines of the book 
by Maclagan and Sturmfels \cite{MacStu}. We now turn to 
tropical geometry as a self-contained subject in its own right.
This is the approach presented
in the book by Mikhalkin and Rau \cite{MikRau}.
Central to that approach is the notion of tropical modification.
In this section we explain how to construct our
tropical del Pezzo surfaces from the 
 plane $\R^2$
by  modifications. This leads to proofs of
Proposition \ref{prop:deg45} and Theorem \ref{thm:deg3}
purely within tropical geometry.

{\em Tropical modification} is an operation that relates topologically different tropical models of the same variety. This operation was first defined by Mikhalkin in \cite{MikICM};
see also \cite[Chapter~5]{MikRau}. Here we work with a variant
known as {\em open tropical modifications}. These were introduced
in the context of Bergman fans of matroids  in \cite{Sha}.
Brugall\'e and Lopez~de Medrano \cite{BL} used them
  to study intersections and inflection points of tropical plane curves.  
 
 We fix a tropical cycle $Y $ in $\R^n$, as in \cite{MikRau}.
 An {\em open modification} is a map $p: Y^{\prime} \rightarrow Y$ where $Y^{\prime} \subset \R^{n+1}$ is  a new tropical variety to be described below. One should think of $Y^{\prime}$ as being an embedding of the complement of a divisor in $Y$ into a higher-dimensional torus. 
  
Consider a piecewise integer affine function $g: Y \rightarrow \R$.  The graph
$$\Gamma_g(Y) \,\,= \,\,\bigl\{ 
(y, g(y)) \ | \ y \in Y \bigr\} \,\,\subset \,\,\R^{n+1}$$ is a polyhedral complex which inherits weights from $Y$. However, it usually not balanced.
There is a canonical way to turn  $\Gamma_g(Y)$
into a balanced complex. If $\Gamma_g(Y)$ is unbalanced
around a codimension one face $E$,  then we attach to it a new unbounded facet $F_E$ 
   in direction $-e_{n+1}$. (We here use the max convention, as in \cite{MikRau}).
         The facet $F_E$ can be equipped with a unique weight $w_{F_E} \in \Z$ such that the 
      complex obtained by adding $F_E$ is balanced at $E$. 
The resulting tropical  cycle is $Y^{\prime} \subset \R^{n+1}$.  By definition,
the {\em open  modification} of $Y$ given by $g$ is the map
$p: Y^{\prime} \rightarrow Y$, where $p$ comes from the projection $\R^{n+1} \rightarrow \R^n$ with kernel~$\R e_{n+1}$. 

The {\em tropical divisor} $\Div_Y(g)$ consists of all points $y \in Y$ such that $p^{-1}(y)$ is infinite. This is a polyhedral complex. It inherits weights on its top-dimensional faces from 
those of $Y^{\prime}$.   A tropical cycle is {\em effective} if the weights of its top-dimensional faces are positive.  Therefore, the cycle $Y^{\prime}$ is effective if and only if $Y$ and the divisor 
$\Div_Y(g)$ are effective.  Given a tropical variety $Y$ and an effective divisor
 $\Div_Y(g)$, we say the 
{\em tropical modification} $p: Y^{\prime} \rightarrow Y$ is {\em along}  $\Div_Y(g)$.
See \cite{MikICM, MikRau, Sha} for basics
concerning cycles, divisors and modifications. 

Open tropical modifications are related to re-embeddings of classical varieties 
as follows.
Fix a very affine $K$-variety $X \subset (\bar{K}^*)^n$
 and $Y = {\rm trop}(X) \subset \R^n$.
Given a polynomial function $f \in K[X]$, let
$D$ be its divisor in $X$. Then $X \backslash D$ is isomorphic
to the graph of the restriction of $f$ to $X \backslash D$.
In this manner, the function $f$ gives a closed embedding
 of $X \backslash D$ into $(\bar K^*)^{n+1}$.
 
For the next proposition we require the tropicalization of a variety to be {\em locally irreducible}. 
Let $y$ be a point in a  tropical variety $Y$, then  
$$\text{Star}_y(Y) \,\,\,=\,\,\, \{  y^{\prime} \ | \ \exists \   \epsilon > 0 \text{ s.t. } \forall  \ 0< \delta< \epsilon
\,:\,  y + \epsilon y^{\prime} \in Y  \},$$
 is a balanced tropical fan  with weights inherited from $Y$.  
 A tropical variety $Y$ is {\em locally irreducible}
  if at every point $y \in Y$, we have that $\text{Star}_y(Y)$ is 
 not a proper union of two tropical varieties,
 taking weights into consideration.


\begin{proposition}\label{prop:mod}
Let $X \subset (\bar{K}^*)^n$ be a very affine variety.
For a function $ f \in K[X]$, let $D$ be the divisor ${\rm div}_X(f)$,
and let $X' = X\backslash D \subset (\bar K^*)^{n+1}$ denote the graph of $X$ along $f$ as described above. 
Let $Y = {\rm trop}(X) \subset \R^n$ and
$  Y' = {\rm trop}(X') \subset \R^{n+1}$. Suppose that $Y$ is locally irreducible, then 
there exists a piecewise integer affine function  $g:Y \rightarrow \R$
such that ${\rm div}_Y(g)= {\rm trop}(D)$ and
the coordinate projection $Y' \rightarrow Y$ is
the open modification of $Y$ along that divisor. \end{proposition}

\begin{proof}
The coordinate projection $p: \R^{n+1} \rightarrow \R^n$ takes $Y'$ onto $Y$, since
tropicalization acts coordinate-wise. 
We claim that the fiber over a 
point $y \in Y$ is either a single  point or a half-line in the $-e_{n+1}$ direction. The fiber, $p^{-1}(y)$ 
is $1$-dimensional   and closed  in $Y^{\prime}$.
 Let $y^{\prime}$ be an endpoint  of a connected component of $p^{-1}(y)$.
 Then  $p(\text{Star}_{y^{\prime}}(Y^{\prime}))$ has the same dimension 
 as $Y$. Since otherwise, $\text{Star}_{y^{\prime}}(Y^{\prime})$ contains a space of linearity in the direction $e_{n+1}$ and $y^{\prime}$ cannot be an endpoint of the fiber. 
If the one dimensional fiber $p^{-1}(y)$ contains two endpoints $y_1$ and $y_2$ then $Y$ must be reducible at $y$; it can be split into more than one component coming from 
the projection of  $p(\text{Star}_{y_1}(Y^{\prime}))$ and  $p(\text{Star}_{y_2}(Y^{\prime}))$. Therefore, $p^{-1}(y)$ consists of either a single point, a line, or a half line. 
However, since $f \in K[X]$ is a regular function, the fiber of a point $y \in Y$ cannot be unbounded in the $+e_{n+1}$ direction. Thus the only possibilities are that $p^{-1}(y)$ is a single point or a half line in the $-e_{n+1}$ direction.

Finally, we obtain the piecewise integer affine function $g$ by taking  $g(y) = p^{-1}(y)$ for $y \in Y \backslash {\rm trop}(D)$ and then extending by continuity to the rest of $Y$. Then $Y^{\prime}$ is the modification along the function $g$  described above. 
\end{proof}

Any two tropical rational functions $g$ and $g^{\prime}$ 
that define the same tropical divisor on $Y$ must differ by a map which is  integer affine on $Y$,
 see \cite[Remark 3.6]{AllRau}.
 This leads to the following corollary.  
 
\begin{corollary} Under the assumptions of Proposition \ref{prop:mod},
 the tropicalization of $X^{\prime} = X \backslash D \subset \R^{n+1}$ is determined uniquely by those
of $D$ and $X$, up to an integer affine map.
\end{corollary}

In general, ${\rm trop}(X^{\prime})$ is not determined by the tropical hypersurface 
of $f \in K[X]$, as the tropicalization of the divisor $D = {\rm div}_X(f)$ may differ from the tropical 
stable intersection of ${\rm trop}(X^{\prime})$ and that tropical hypersurface. 
 Examples 4.2 and 4.3 of \cite{BL} demonstrate both this and that
 Proposition   \ref{prop:mod} can fail without   the locally  irreducibility hypothesis. 

Suppose now that $X^{\prime} \subset (\bar K^*)^{n+k}$ is obtained from 
$X \subset  (\bar K^*)^{n}$ by taking the graph of
a list of $k \geq 2$ polynomials $f_1,f_2, \ldots, f_k$.
This gives us a  sequence of projections
\begin{equation}
\label{eq:projsequence}
 X' = X_k \rightarrow X_{k-1} \rightarrow \cdots \rightarrow X_2 \rightarrow X_1 \rightarrow
X_0 = X , 
\end{equation}
where $X_i \subset (\bar K^*)^{n+i}$ is obtained
from $X$ by taking the graph of $(f_1,\ldots,f_i)$.
We further get a corresponding sequence of projections of the tropicalizations:
\begin{equation}
\label{eq:modisequence}
 {\rm trop}(X^{\prime}) = {\rm trop}(X_k) \rightarrow {\rm trop}(X_{k-1}) \rightarrow \cdots  \rightarrow
{\rm trop}(X_0) = {\rm trop}(X).
\end{equation}
We may ask if it is possible to recover  ${\rm trop}(X^{\prime}) \subset \R^{n+k}$ just from
  ${\rm trop}(X)$ and the  $k$ tropical divisors  ${\rm trop}(D_i)$ considered in ${\rm trop}(X)$. 
The answer is ``yes'' in the special case when
  ${\rm trop}(X) = \R^n$ and the arrangement of divisors $ {\rm trop}(D_i)$ are each locally irreducible and intersect properly, meaning the intersection of any $l$ of these divisors has codimension $l$.  
However, in general,  iterating modifications to recover ${\rm trop}(X^{\prime})$ is a delicate procedure.
In most cases, the outcome is not solely determined by the configuration of tropical divisors in ${\rm trop}(X)$, even if the divisors intersect pairwise properly.
We illustrate this by deriving the degree $5$ del Pezzo surface~${\rm trop}(M_{0,5})$.

\begin{figure}
\begin{center}
\includegraphics[scale=2.5]{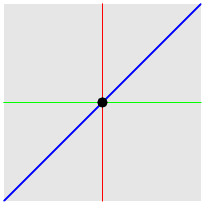}
\hspace{1cm}
\includegraphics[scale=1.5]{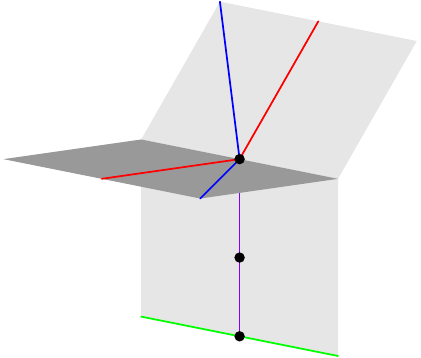}
\end{center}
\vspace{-0.2in}
\caption{The tropical divisors in Example \ref{ex:modifytopetersen}. The positions of ${\rm trop}(G_1 \cap H_1)$ in ${\rm trop}(X_1)$ for three choices of $a$  are marked on the
downward purple edge.
For $a  = 1$ we get $M_{0,5}$.
\label{fig:M05mod1}}
\end{figure}

\begin{example} \label{ex:modifytopetersen} \rm
This is a variation on \cite[Example 2.29]{Sha}. 
Let $X = (\bar{K}^*)^2$ and consider the functions
$f(x) = x_1 - 1$, $g(x) = x_2 - 1$ and $h(x) = ax_1 - x_2$, 
for some constant $a \in K^*$ with ${\rm val}(a) = 0$.
Denote ${\rm div}_X(f)$ by $F$, and analogously for $G$ and $H$. 
The tropicalization of each divisor is a line through the origin in $\R^2$. 
The directions of ${\rm trop}( F), {\rm trop}(G)$, and ${\rm trop}(H)$ are $(1, 0), (0, 1),$ and $(1, 1)$ respectively.
Let $X^{\prime} \subset  (\bar{K}^*)^5$ denote the graph of $X$ along the three functions $f, g,$ and  $h$, in that order.  
This defines a sequence of projections, $$X^{\prime} \longrightarrow X_2 \longrightarrow X_1 \longrightarrow X = (\bar{K}^*)^2.$$
Here,
$\,X_2 = \{ (x_1, x_2, x_1-1, x_2-1) \} \subset (\bar{K}^*)^4$.
The tropical plane $\mathrm{trop}(X_2)$ contains the face
$\sigma= \{0\}\times{}\{0\}\times{}(-\infty{},0]\times{}(-\infty{},0]$,
corresponding to points with
$\mathrm{val}(x_1)=\mathrm{val}(x_2)=0$.
Let  $H_2 $ denote the graph of $f$ and $g$ restricted to $H$. This is a line in 
$4$-space, namely,
$$H_2 \,=\, \{ (x_1, ax_1, x_1 - 1, ax_1 -1)  \}\, \subset \,X_2 \,\subset\, (\bar{K}^*)^4.$$
The tropical line $\mathrm{trop}(H_2)$ depends on the valuation
of $a-1$. It can be determined from 
$${\rm trop}(G_1 \cap H_1) \,=\,
\bigl\{ \bigl(0,0,-\mathrm{val}(\frac{1}{a}-1)\bigr)\bigr\}.$$ 
Here, $H_1, G_1$ denote the graph of $f$ restricted to $H$ and $G$, respectively.
 Figure \ref{fig:M05mod1}  shows the possibilities 
for ${\rm trop}(G_1 \cap H_1)$ in ${\rm trop}(X_1)$, and Figure \ref{fig:M05mod2}  shows
$\mathrm{trop}(H_2)\cap  \sigma$ in $ {\rm trop}(X_2)$.

\begin{figure}
\begin{center}
\includegraphics[scale=1.7]{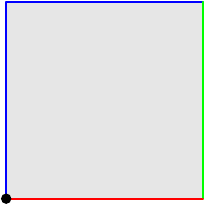}
\put(-70, -15){$v = 0$}
\hspace{1cm}
\includegraphics[scale=1.7]{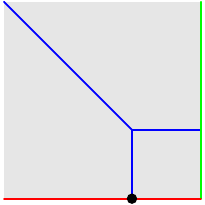}
\put(-83, -15){$v \in (0,\infty)$}
\hspace{1cm}
\includegraphics[scale=1.7]{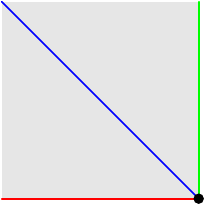}
 \put(-70, -14){$v = \infty $}
 \end{center}
 \vspace{-0.2in}
\caption{The different possibilities for ${\rm trop}(H_2) \cap \sigma$
 in Example \ref{ex:modifytopetersen}
 \label{fig:M05mod2} }
\end{figure}

We can prescribe any value
$v \in ( 0,\infty)$ for the
valuation of $\frac{1}{a} -1 $, for instance by taking
 $\frac{1}{a} = 1 + t^v$ when $K = \C \{\!\{t\}\!\}$. In these cases,
 the tropical plane ${\rm trop}(X')$ is not a fan.
However, it becomes a fan when $v$ moves to either
endpoint of the interval $[0,-\infty]$. For instance, $v = 0$ happens when
 the constant term of $\frac{1}{a}$ is not equal to $1$ and ${\rm trop}(X')$ is the fan  obtained from $\R^2$ by
 carrying out the modifications along the \emph{pull-backs}   of the    tropical divisors
   on the left in Figure \ref{fig:M05mod1}. See  Definition 2.16 of \cite{Sha} for pull-backs of tropical divisors. 
The other extreme is when $a=1$. Here, 
$F, G,H$ are concurrent lines in $(\bar K^*)^2$,
and   ${\rm trop}(H_2)$ contains a ray in the direction $e_3+e_4$. 
 Upon modification, we obtain the fan over the Petersen graph in Figure \ref{fig:petersen}. This is the tropicalization of the 
degree $5$ del Pezzo surface
 in (\ref{embeddeg5}).
 Thus beginning from the tropical divisors ${\rm trop}( F), {\rm trop}(G)$, and ${\rm trop}(H)$ in 
$\R^2$,  we recover  ${\rm trop}(M_{0,5})$ if we know that they represent tropicalizations of concurrent lines in $(\bar{K}^*)^2$. 

The open tropical modification described above represents the tropicalization of the very affine variety $M_{0, 5}$. 
In this case the  compactification of $M_{0,5}$ which produces the del Pezzo surface of degree $5$ is indeed the tropical compactification 
given by the fan ${\rm trop}(M_{0,5})$, (see \cite[\S 6.4]{MacStu} for an introduction to tropical compactifications). 
There is no direct connection
between open tropical modifications and birational transformations, the link  depends a choice of compactification of the very affine variety. 
Upon removing divisors one can find  more interesting compactifications of the complement. 
For example, $(K^*)^2$ cannot be compactified to a del Pezzo surface of degree less than $6$, but upon deleting the three divisors above one can compactify the complement 
to  a del  Pezzo surface of degree $5$. 
 \hfill $\diamondsuit$
\end{example}


We now explain how this extends to a del Pezzo surface $X$
 of degree $d \leq 4$.
As before, we write $X^0$ for the complement of
the $(-1)$-curves in $X$. Then  $X' = X^0$ is obtained from 
$(\bar{K}^*)^2$ by taking the
graphs of the polynomials  $f_1,\ldots, f_k$ of the curves in $(\bar K^*)^2$
 that give rise to  $(-1)$-curves on $X$. 
More precisely, fix $p_1 = (1:0:0)$, $p_2= (0:1:0)$ , $p_3 = (0:0:1)$,
$p_4 = (1:1:1)$, and take $p_5,\ldots,p_{9-d}$ 
to be general points in $\PP^2$.
If $d = 4$ then there is only one extra point $p_5$, we have $k = 8$ in
(\ref{eq:projsequence}), and $f_1,\ldots,f_8$ are the  polynomials  defining
\begin{equation}
\label{eq:8fordeg4}
 F_{14}, F_{15}, F_{24},\, F_{25}, F_{34}, F_{35}, \, F_{45}, \,G. 
 \end{equation}
For $d=3$, there are two extra points $p_5,p_6$ in $X$, we have $k = 18$,
and $f_1,\ldots,f_{18}$ represent
\begin{equation}
\label{eq:8fordeg3}
 F_{14}, F_{15}, F_{16},\,
  F_{24}, F_{25}, F_{26},\,
   F_{34}, F_{35},F_{36},\,
   F_{45}, F_{46},F_{56}, \,
   G_1,G_2,G_3,G_4,G_5,G_6.
 \end{equation}

\begin{figure}
\begin{center}
\includegraphics[scale=0.9]{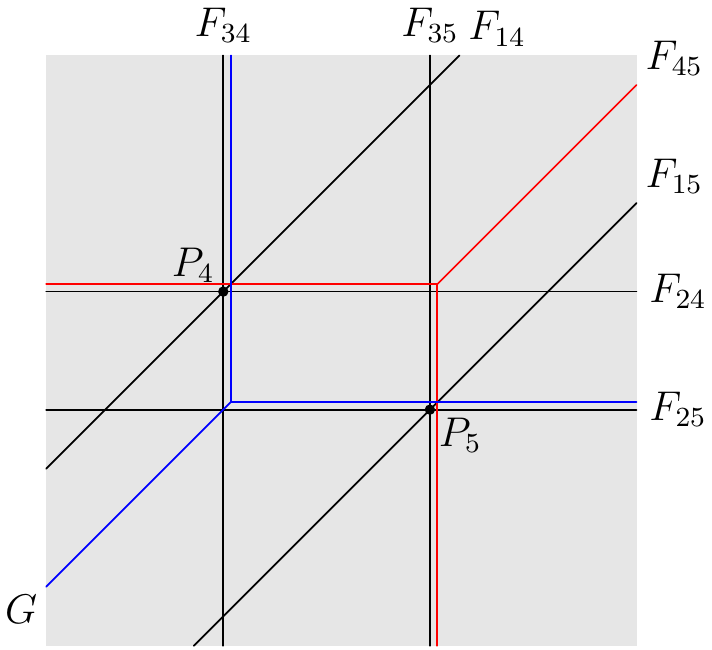}
\hspace{2cm}
\includegraphics[scale=1]{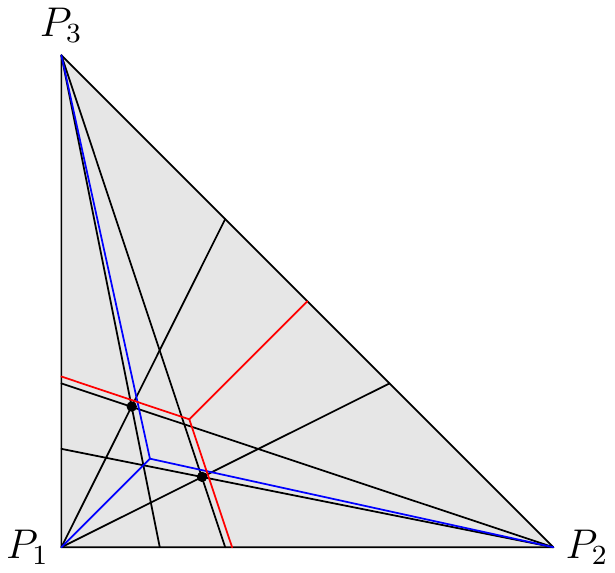}
\end{center}
\vspace{-0.3in}
\caption{The tropical conic and the tropical lines determined by the $5$ points  for a
marked del Pezzo surface of  degree $4$. The diagram is
drawn in $\R^2$ on the left and in $\TP^2$ on the right. 
The $16$ trivalent trees corresponding to the $(-1)$-curves of the del Pezzo surface,
seen at the nodes in in Figure~\ref{fig:clebsch}, arise
from the plane curves shown here by tropical modifications.
\label{fig:curveconfig}
}
\end{figure}

We write $P_i = {\rm trop}(p_i) \in \TP^2$ for the image
of the point $p_i$ under tropicalization.
The tropical points $P_1,P_2,\ldots$ are in {\em general position}
if any two lie in a unique tropical line, these lines are distinct,
any five lie in a unique tropical conic, and these conics are distinct in $\TP^2$.
A configuration in general position for $d=4$ is shown in
Figure \ref{fig:curveconfig}. Our next result
implies that the colored Clebsch graph in  Figure \ref{fig:clebsch}  can be read off from
Figure \ref{fig:curveconfig} alone.
For $d=3$, in order to recover the tropical cubic surface
from the planar configuration, the points $P_i$
must satisfy further genericity assumptions, to be
 revealed in the proof of the next theorem. 
 
\begin{theorem} \label{thm:modify}
Fix $d \in \{3,4,5\}$ and
 points $p_1, \ldots,p_{9-d}$ in $\PP^2$ whose
 tropicalizations $P_i$ are sufficiently generic in  $\TP^2$.
The  tropical  del Pezzo surface ${\rm trop}(X^0)$ can be
constructed from $\TP^2$ by
a sequence of open modifications  that is determined by
the points $P_1, \ldots, P_{9-d}$.
\end{theorem}

 \begin{proof}
The sequence of tropical modifications we use to go from $\R^2$
to ${\rm trop}(X^0)$ is determined if we know, for each $i$,
the correct divisor on each   $(-1)$-curve $C$
in the tropical model ${\rm trop}(X_i)$.
Then, the preimage of $C$ in the next surface ${\rm trop}(X_{i+1})$
is the modification $C'$ of the curve $C$ along that  divisor.
By induction,  each intermediate surface ${\rm trop}(X_{i})$ is locally irreducible, 
since it is obtained by modifying a  locally irreducible surface along a locally irreducible 
divisor. 
With this, Theorem~\ref{thm:modify}
follows from Proposition \ref{prop:mod}, applied to  both
the $i$-th surface and its $(-1)$-curves.
The case $d=5$ was covered  in Example~\ref{ex:modifytopetersen}. 
From the metric tree that represents the boundary divisor $C$ of $X^0$ 
we can derive  the corresponding trees on each intermediate surface ${\rm trop}(X_i)$
by  deleting leaves. Thus, to establish
Theorem~\ref{thm:modify}, it suffices to prove the following claim:
{\em 
the final arrangement of the ($16$ or $27$) metric trees on
the tropical del Pezzo surface
 ${\rm trop}(X^0)$
is  determined by the locations of the points $P_i$ in~$\TP^2$.}

Consider first the case $d=4$. The points
 $P_4$ and $P_5$ determine an arrangement
of plane tropical curves (\ref{eq:8fordeg4}) as shown in
Figure \ref{fig:curveconfig}. The conic $G$
through all five points looks like an ``inverted tropical line'',
with three rays in directions $P_1,P_2,P_3$. By the genericity assumption,
the points $P_4$ and $P_5$ are located on distinct rays of $G$.
These data determine a trivalent metric tree with five leaves, which we now label
by $E_1,E_2,E_3,E_4,E_5$. Namely, $P_4$ forms a cherry together
with the label of its ray, and ditto for $P_5$. For instance,
in Figure \ref{fig:curveconfig}, the cherries on the tree $G$
are $\{E_1,E_4\}$ and $\{E_2,E_5\}$, while $E_3$ is the
non-cherry leaf. This is precisely the tree sitting on the node
labeled $G$ in Figure \ref{fig:clebsch}. The lengths of the
two bounded edges of the tree $G$
are the distances from $P_4$ resp.~$P_5$ to the unique
vertex of the conic $G$ in $\RR^2$.
Thus the metric tree $G$ is easily determined from 
$P_4$ and $P_5$.
The other $15$ metric trees  can also be determined in a similar way 
from the configuration of points and curves in $\R^2$ and by performing a subset of the 
necessary modifications. Alternatively, we may use the transition rules
(\ref{eq:relabel1}) and (\ref{eq:relabel2}) to obtain the other $15$ trees from $G$.
This proves the above claim, and hence Theorem \ref{thm:modify},
for del Pezzo surfaces of degree $d=4$.

Consider now the case $d=3$. Here the arrangement of tropical plane curves
in $\R^2 \subset \TP^2$ consists of
three lines at infinity, $F_{12},F_{13},F_{23}$,
nine straight lines, $F_{14}, F_{15}, \ldots, F_{36}$,
three honest tropical lines, $F_{45}, F_{46},F_{56}$,
three conics that are ``inverted tropical lines'' $ G_4, G_5, G_6$,
and three conics with one bounded edge, $G_1,G_2,G_3$.
Each of these looks like a tree already in the plane, and it
gets modified to a $10$-leaf tree, like to ones in
Figures \ref{fig:treeaaaa} and~\ref{fig:treeaaab}.
We claim that these labeled metric trees are uniquely
determined by the positions of $P_4,P_5,P_6$ in $\R^2$.

Consider one of the $9$ straight lines in our arrangement, say, $F_{14}$.
If the points $P_4, P_5, P_6$ are 
generically chosen,
  $7$ of the $10$ leaves on the tree  $F_{ij}$ can be determined
from the diagram in $\R^2$. These come from the $7$ markings 
on the  line $F_{14}$ given~by
$ E_1, E_4,  F_{23}, F_{25}, F_{26},  F_{35}, F_{36}. $
The markings $E_1$ and $F_{23}$ are the points at infinity,
 the marking $E_4$ is the location of point $P_4$, and
the markings $F_{25}, F_{26},  F_{35}, F_{36}$ are the
points of intersection with those lines. Under our hypothesis,
these  $7$ marked points
on the line $F_{14}$ will be distinct. 
With this, $F_{14}$ is 
already a metric caterpillar tree with $7$ leaves.
The three markings which are missing are $G_1, G_4$ and $F_{56}$. 
Depending on the positions of  $P_4, P_5, P_6$,
the intersection points  of these three curves with the line $F_{14}$
may coincide with previously marked points. Whenever this happens,
 the position of the additional marking on the tree $F_{14}$ can be anywhere on the already attached leaf ray. Again, the actual position of the point on that
 ray may be determined by performing modifications along those curves. Alternatively, we use the involution
given in Corollary \ref{prop:invol}. The involution on the ten leaves of the desired tree $F_{14}$ is
$$ 
E_1 \leftrightarrow \underline{G_4}, \quad
E_4 \leftrightarrow \underline{G_1} ,\quad
F_{23} \leftrightarrow  \underline{F_{56}}, \quad
F_{25} \leftrightarrow F_{36} ,\quad
 F_{26} \leftrightarrow  F_{35}.
 $$
 Since the involution exchanges each of the three unknown leaves with
 one of the seven known leaves, we can easily construct the final
 $10$-leaf tree from    the $7$-leaf caterpillar.

  A similar argument works the other six lines
  $F_{ij}$, and the conics $G_4, G_5, G_6$. In these cases, $8$ of the $10$ marked points on a
  tree are determined from the arrangement in the plane, provided the choice of points is generic.  
 Finally, the conics $G_1, G_2, G_3$  are dual to subdivisions of lattice parallelograms of area $1$. 
 They may contain a bounded edge. 
 Suppose  no  point $P_j$ lies on the bounded edge of the conic $G_i$, then the positions of all $10$ marked points of the tree are visible from the 
 arrangement in the plane. If $G_i$ does contain a marked point $P_j$ on its bounded edge, then the tropical line $F_{ij}$ intersects $G_i$ in either a bounded edge or a single point with  
 intersection  multiplicity $2$, depending on the dual subdivision of $G_i$. In the first case the position of the marked point $F_{ij}$ is easily determined from the involution; the distance from a vertex of the bounded edge of $G_i$ to the marked point $F_{ij}$ must be equal to the  distance from $P_j$ to the opposite vertex of the bounded edge of $G_i$.  

If $G_i \cap F_{ij}$  is a single point of intersection multiplicity two, then $P_j$ and $F_{ij}$ form a cherry on the tree $G_i$ which is invariant under the involution. We  claim that this cherry attaches to the rest of the tree at a $4$-valent vertex. 
The involution on the $10$-leaf tree can also be seen as a tropical double cover from our
 $10$-leaf tree to a $5$-leaf tree, $h: T \rightarrow t$, where the $5$-leaf tree $t$ is labeled with the pair of markings interchanged by the involution. As mentioned in Corollary \ref{prop:invol},  this double cover comes from the classical curve in the del Pezzo  surface $X$. In particular, the double cover locally satisfies the tropical translation of the  Riemann-Hurwitz condition  \cite[Definition 2.2]{BBM}. In our simple case of a degree $2$ map between two trees, this local condition for a vertex $v$ of $T$ is
${\rm  deg}(v)   - d_{h, v}({\rm deg }(h(v)) - 2) - 2 \geq 0$, where ${\rm deg}$ denotes the valency of a vertex, and $d_{h, v}$ denotes the local degree of the map $h$ at $v$.
Suppose the two leaves did not attach at a four valent vertex, then they form a cherry, this cherry attaches to the rest of the tree by an edge $e$ which is adjacent to another vertex $v$ of the tree. The Riemann-Hurwitz condition is violated at $v$, since ${\rm deg}(v) = {\rm deg}(h(v)) = 3$ and $d_{h, v} = 2$. 

We conclude that
the tree arrangement can be recovered from the position of the points $P_1, P_2, \ldots\,$ in $\R^2$.
Therefore it is also possible
 to recover the tropical del Pezzo surface ${\rm trop}(X^0)$ by open modifications.
  In each case, we  recover the corresponding final $10$ leaf tree from the arrangement
  in $\TP^2$  plus our knowledge of the involution in Corollary \ref{prop:invol}.
  \end{proof}
  
  \begin{remark} \rm
      Like in the case $d=4$,   knowledge of transition rules among the 
  $27$ metric trees on ${\rm trop}(X^0)$ can greatly simplify their reconstruction. We
  give such a rule in Proposition~\ref{prop:inducedtrees}.
\end{remark}

In this section we gave a geometric construction of
tropical del Pezzo surfaces of degree $d \geq 3$,
starting from the points
$P_1,\ldots,P_{9-d}$ in the tropical plane $\TP^2$.
The lines and conics in $\TP^2$ that
correspond to the $(-1)$-curves are transformed,
by a sequence of open tropical modifications, into
the trees that make up the boundary of the del Pezzo surface.
Knowing these well-specified modifications of curves ahead of time
allows us to carry out a unique sequence
of open tropical modifications of surfaces,
starting with $\R^2$. In each step, going from
right to left in (\ref{eq:projsequence}), we modify
the surface along a divisor given by
one of the trees.

This gives a geometric construction
for the bounded complex in a tropical del Pezzo surface:
it is the preimage under (\ref{eq:projsequence})
of the  bounded complex in the arrangement in $\R^2$.
For instance, Figure~\ref{fig:clebsch2} is the preimage
of the parallelogram and the four triangles 
in Figure~\ref{fig:curveconfig}.

The same modification approach can be used to construct
(the bounded complexes of) any tropical plane in $\TP^n$
from its tree arrangement.
This provides a direct link between 
the papers \cite{HJJS} and \cite{Sha}.
That link should be useful for readers
of the text books \cite{MacStu} and \cite{MikRau}.

\section{Tropical Cubic Surfaces and their 27 Trees}
\label{sec:combinatorics}

This section is devoted to the combinatorial structure of tropical cubic surfaces.
 Throughout, $X$ is a smooth del Pezzo surface of degree $3$,
without Eckhart points, and $X^0$ the very affine surface obtained 
by removing  the $27$ lines from $X$. Recall that an {\em Eckhart point}
  is an ordinary triple point in the 
  union of the $(-1)$-curves.
 Going well  beyond the
summary statistics of Theorem \ref{thm:deg3}, we now offer
an in-depth study of the combinatorics of
the surface ${\rm trop}(X^0)$.

We begin with the construction of $\trp(X^0)$ 
from six points in $\TP^2$, as in
Section \ref{sec:mod}. The points
$P_5$ and $P_6$ are general in
$\R^2 \subset \TP^2$. The first four points are
the coordinate points
\begin{equation}
\label{eq:ourchoi}
P_1 = (0:-\infty:-\infty), \,\,
P_2 = (-\infty:0:-\infty), \,\,
P_3 = (-\infty:-\infty:0), \,\,
P_4 = (0:0:0). 
\end{equation}
Theorem  \ref{thm:modify} tells us
that  $\trp(X^0)$ is determined by the locations of
$P_5$ and $P_6$ when the points are generically chosen. 
There are two generic types, namely (aaaa) and (aaab), as shown 
in Figures \ref{fig:treeaaaa} and \ref{fig:treeaaab}. This raises
the question of how the type can be decided
from the positions of
$P_5$ and $P_6$. To answer that question, we shall
use {\em tropical convexity} \cite[\S 5.2]{MacStu}.
There are five generic types of tropical triangles,
depicted  here in Figures \ref{fig:trianglesaaab} and \ref{fig:trianglesaaaa}. 
The unique $2$-cell in such a tropical triangle has either $3,4,5$ or $6$ vertices.
Two of these have $4$ vertices, but only one type contains a parallelogram.
That is the type which gives (aaaa).

\begin{theorem} \label{thm:parallelogram}
Suppose that the tropical cubic surface 
constructed as in Theorem \ref{thm:modify}
has one of the two generic types.
Then it has  type (aaaa) if and only if the $2$-cell in 
the tropical triangle spanned by $P_4, P_5$ and $ P_6$ is a parallelogram. 
In all other cases,  it has type (aaab).
\end{theorem}

Note that the condition that the six points $P_i$ are in general position is not sufficient to imply that the tropical cubic surface is generic. In some cases, the corresponding point in the Naruki fan $\mathrm{trop}(\mathcal{Y}^0)$ will lie on the boundary of the subdivision induced by the map from $\mathrm{trop}(\mathcal{G}^0)$, as described in Section \ref{sec2}
and below. If so, the tropical cubic surface is degenerate.

\begin{figure}
\begin{center}
\includegraphics[scale=0.4]{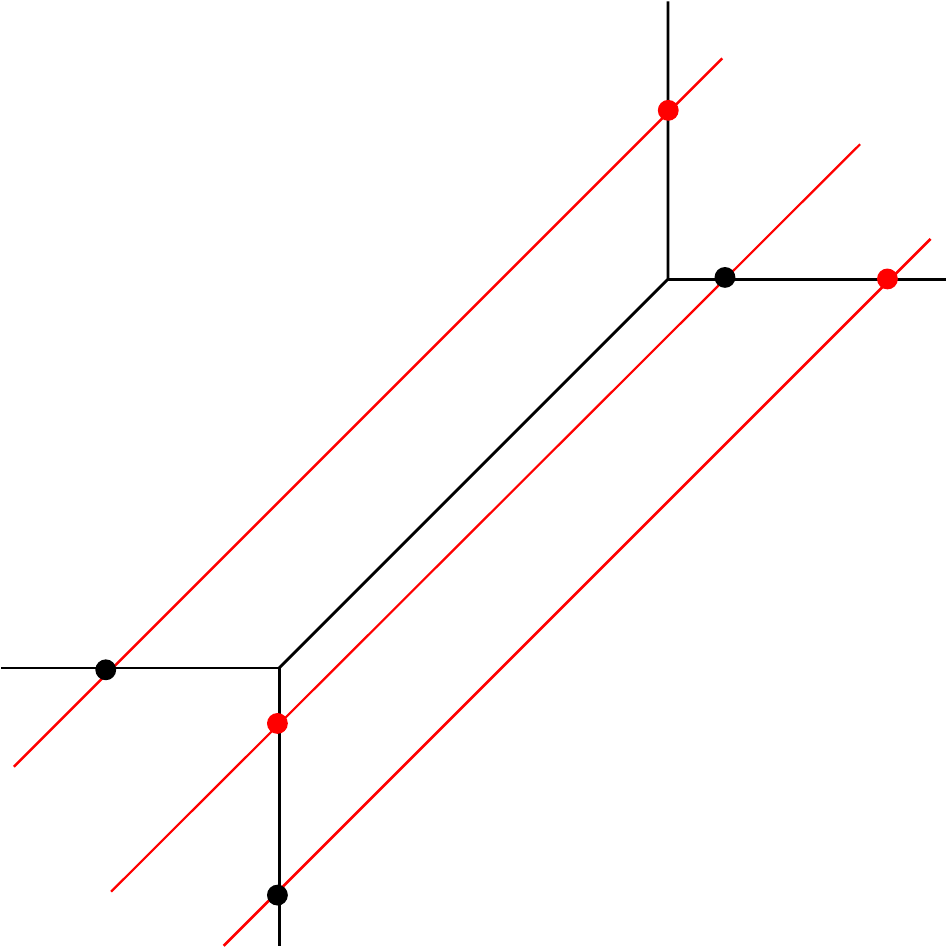}\qquad
\includegraphics[scale=0.4]{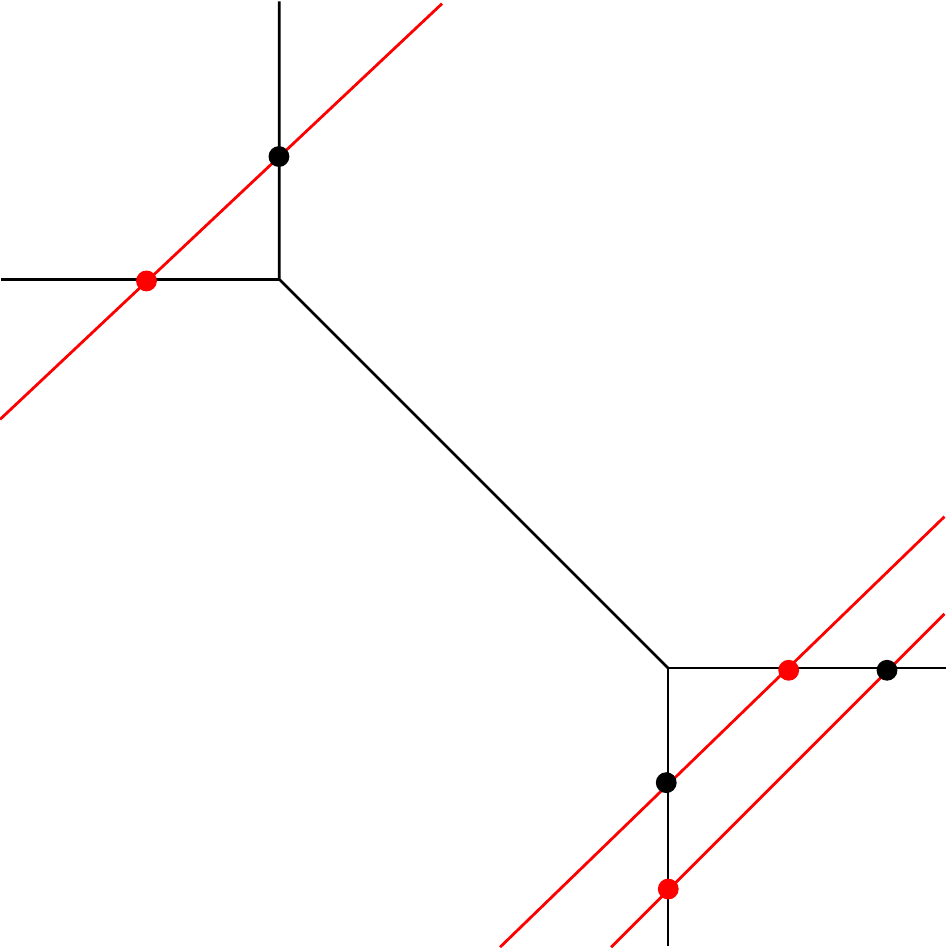}\qquad 
\includegraphics[scale=0.4]{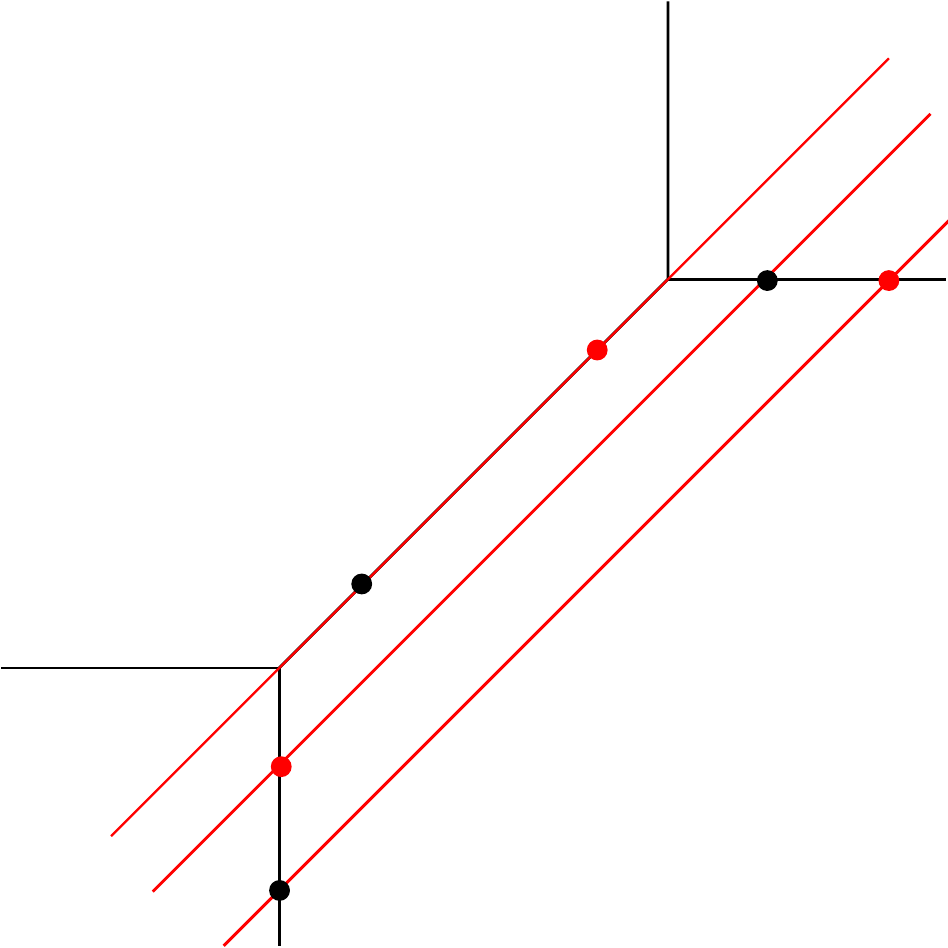} 
\end{center}
\vspace{-0.25in}
\caption{Markings of a conic $G_1$ which produce trees of type (aaab).
 \label{fig:markedConics1}}
 \medskip
\end{figure}

\begin{figure}
\begin{center}
\includegraphics[scale=0.5]{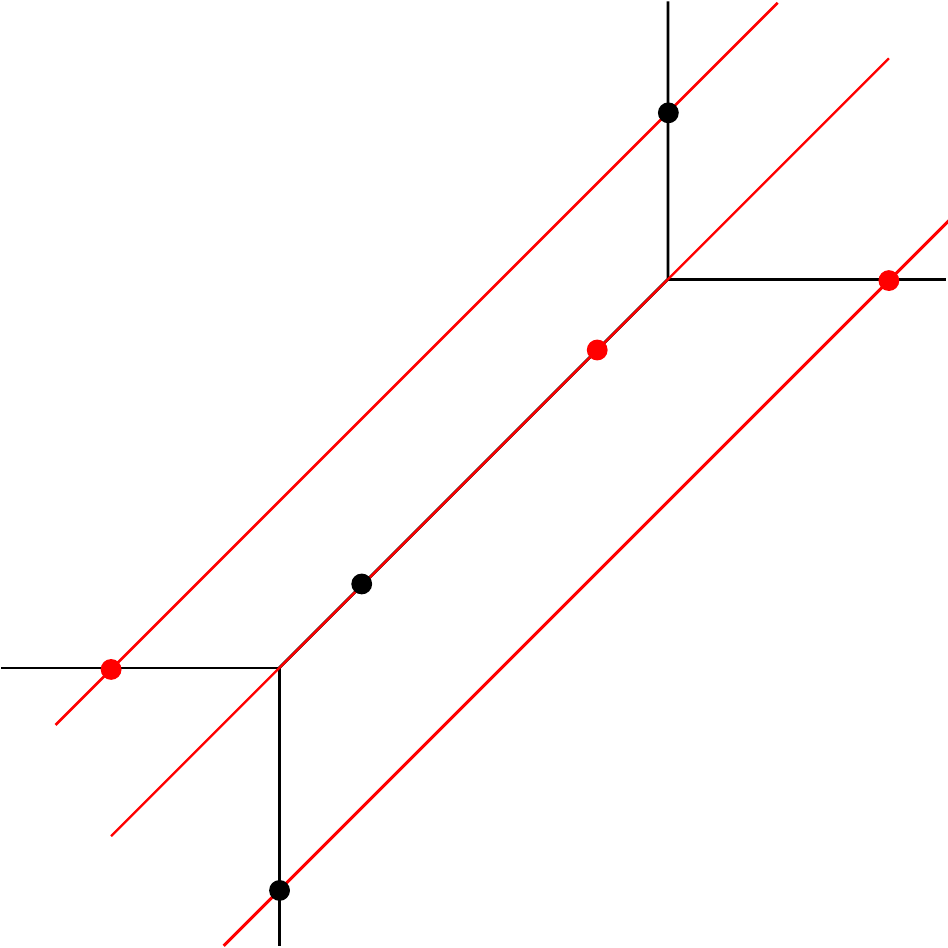} \qquad  \qquad  \qquad
\includegraphics[scale=0.5]{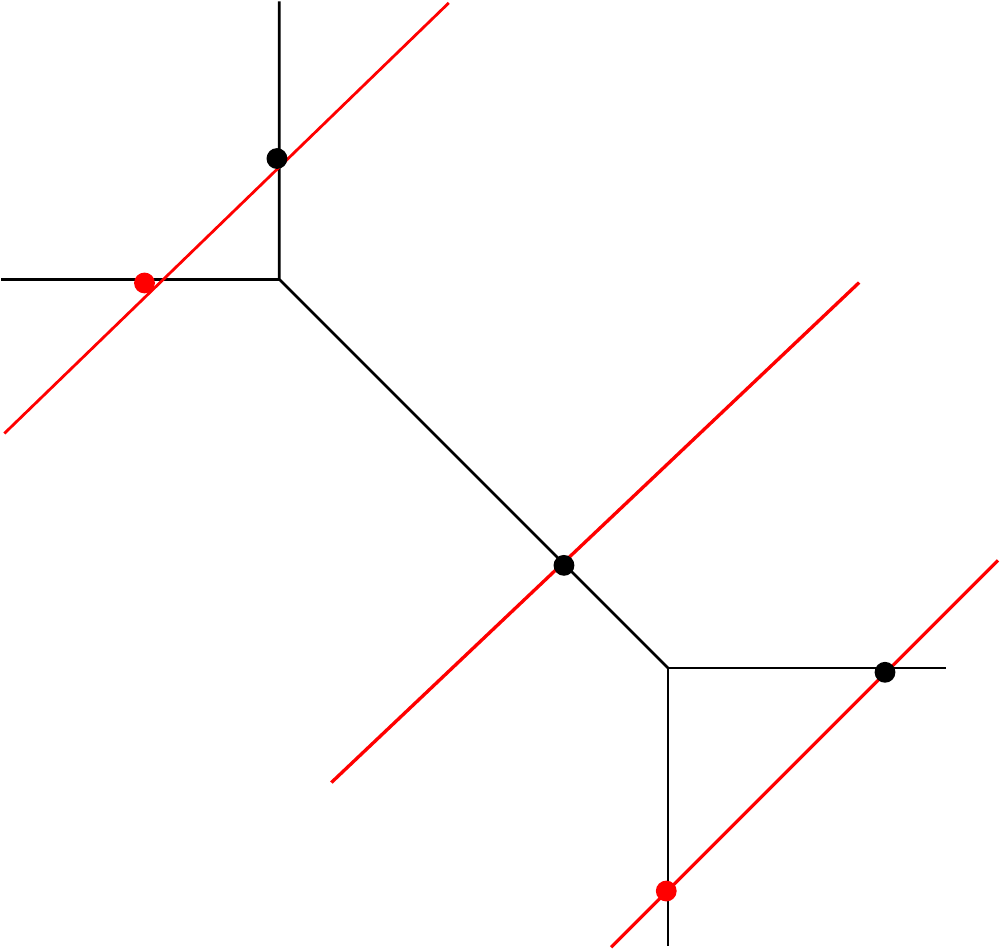} 
\put(-80, 55){\small{$2$}}
\end{center}
\vspace{-0.3in}
\caption{Markings of a conic $G_1$ which produce trees of type $(aaaa)$.
 \label{fig:markedConics2}}
\end{figure}

\begin{proof}[Proof of Theorem~\ref{thm:parallelogram}]

The tree arrangements for the  two types of generic surfaces
consist of distinct combinatorial types, i.e.~there is no overlap in
Figures  \ref{fig:treeaaaa} and  \ref{fig:treeaaab}.
    Therefore, when the tropical cubic surface is generic, it is enough to determine the combinatorial type of a single tree. We do this for the 
    conic $G_1$. 
Given our choices of points  (\ref{eq:ourchoi}) in $\TP^2$,
the tropical conic $G_1$ is dual to the Newton polygon with vertices $(0, 0), (1, 0), (0, 1)$, and $(1, 1)$. 
The triangulation has one interior edge, either of slope $1$ or of slope $-1$. 
We claim the following:
 
 \medskip
 
{\em
The tropical cubic surface ${\rm trop}(X^0)$ has type (aaaa) if and only if the following holds:
\begin{enumerate}
\item The bounded edge of the  conic $G_1$ has slope $-1$ and contains a marked point $P_j$, or \item the bounded edge of the conic $G_1$ has slope $1$ and contains a marked point $P_j$, and the other two points $P_j, P_k$ lie on opposite sides of
 the line spanned by the bounded edge.
 \end{enumerate}
}

\begin{figure}[b]
\begin{center}
\vspace{-0.15in}
\includegraphics[scale=0.5]{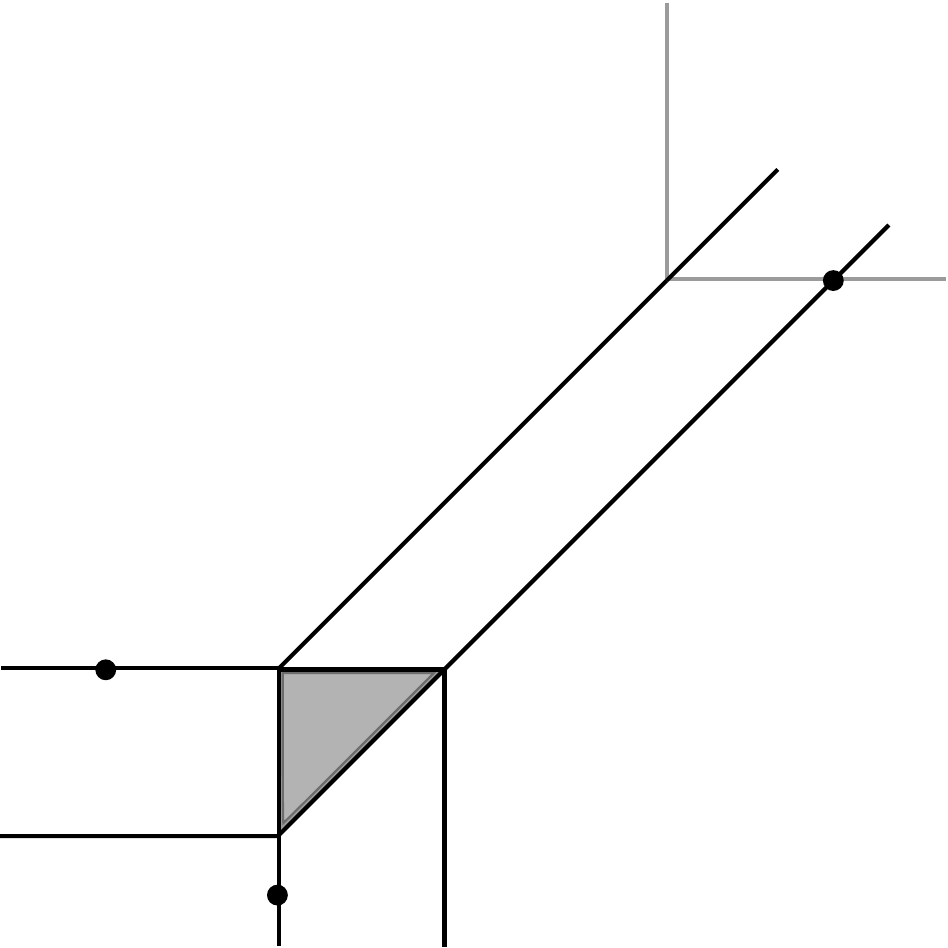}\qquad
\includegraphics[scale=0.5]{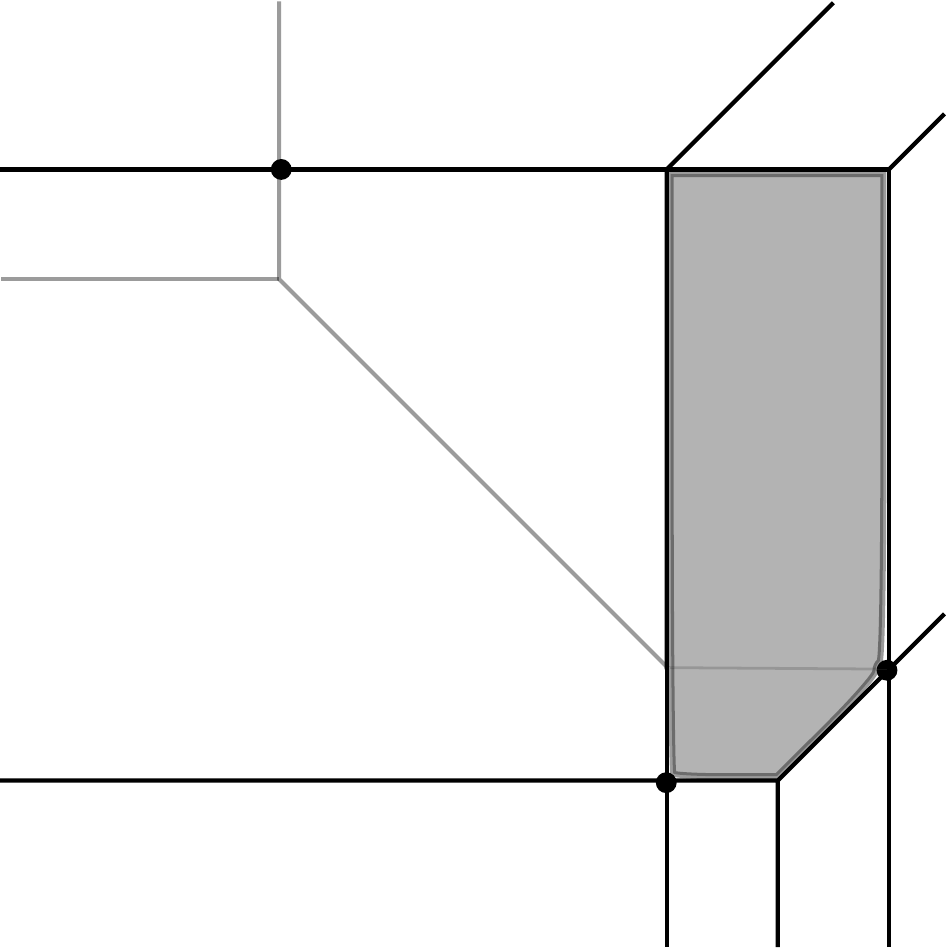}\qquad
\includegraphics[scale=0.5]{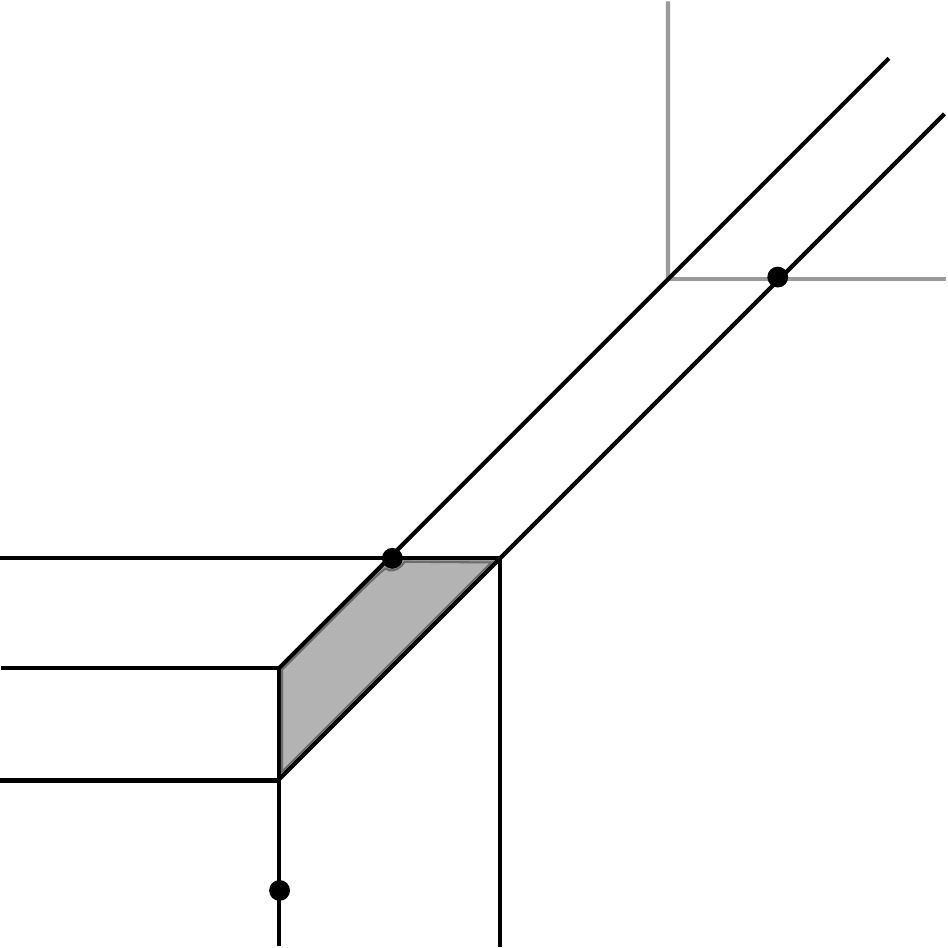}
\vspace{-0.13in}
\caption{The tropical triangles formed by points on $G_1$ as in Figure \ref{fig:markedConics1},
giving type (aaab).
\label{fig:trianglesaaab}}
\end{center}
\end{figure}

\begin{figure}
\begin{center}
\vspace{-0.13in}
\includegraphics[scale=0.56]{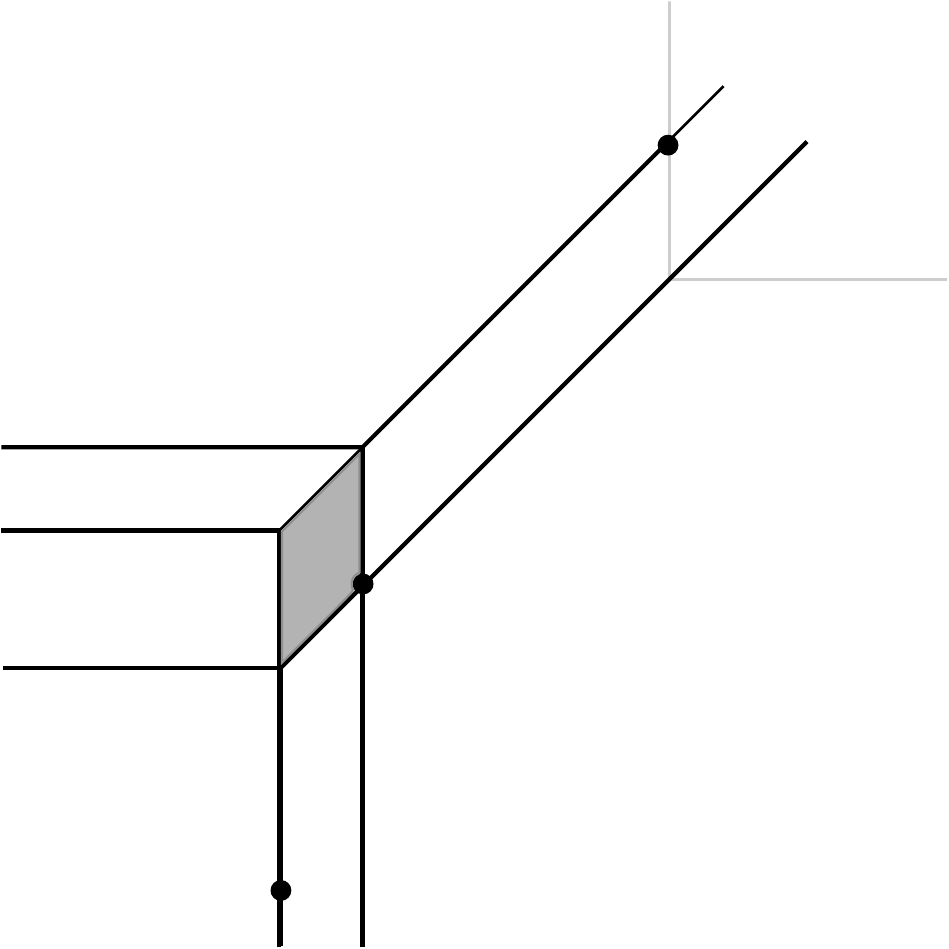}\qquad\qquad
\includegraphics[scale=0.5]{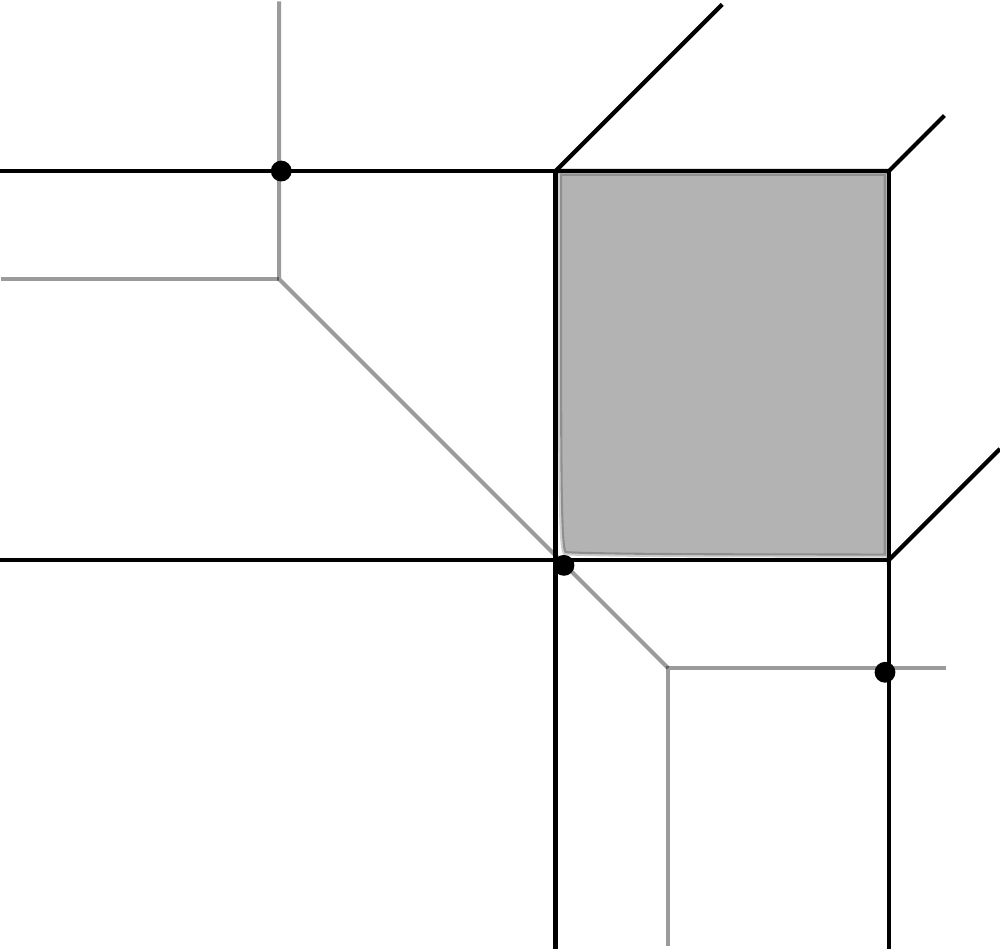}
\vspace{-0.13in}
\caption{The tropical triangles formed by points on $G_1$ as in Figure \ref{fig:markedConics2},
giving type (aaaa).
\label{fig:trianglesaaaa}}
\end{center}
\end{figure}

To show this, we follow the proof of Theorem \ref{thm:modify}. For each configuration of $P_4$, $P_5$, $P_6$ on the conic $G_1$, we draw lines
with slope $1$ through these points. These are the tropical lines $F_{14}$, $F_{15}$, $F_{16}$. 
Each intersects $G_1$ at one further point. These are the images of $E_4,E_5,E_6$ under the tree involution, {\it i.e.}~the points labeled $F_{14}$, $F_{15}$, $F_{16}$ on the tree $G_1$. Together with $E_2$, $E_3$, $F_{12}$ and $F_{13}$ lying at infinity of $\mathbb{TP}^2$, we can reconstruct a tree with $10$ leaves. Then, we can identify the type of the tree arrangement. We did this for all possible configurations up to symmetry. Some of the results are shown in Figures \ref{fig:markedConics1} and 
 \ref{fig:markedConics2}. The claim follows.
 
To derive the theorem from the claim, we must consider the tropical convex hull
 of the points $P_4, P_5, P_6$ in the above cases. 
As an example, the $2$-cells of the tropical triangle 
 corresponding to the trees in Figures \ref{fig:markedConics1} and \ref{fig:markedConics2} are shown in Figures \ref{fig:trianglesaaab} and \ref{fig:trianglesaaaa} respectively.  
The markings of $G_1$ producing a type (aaaa) tree always give parallelograms. Finally, if the marking of a conic produces a type (aaab) tree then the $2$-cell may have $3, 4, 5,$ or $6$ vertices. However, if it has $4$ vertices, then it is a trapezoid with only one pair of parallel edges.
\end{proof}

We next discuss some relations among the $27$ boundary trees  of a 
tropical cubic surface~$X$. Any pair of  disjoint $(-1)$-curves on $X$ meets exactly five other $(-1)$-curves. Thus, two $10$-leaf trees $T$ and $T^{\prime}$ representing disjoint $(-1)$-curves have exactly five leaf labels in common. Let $t$ and $t'$ denote the $5$-leaf trees 
constructed from $T$ and $T^{\prime}$ as in the proof of Proposition \ref{prop:invol}.
Thus $T$ double-covers $t$, and $T'$ double-covers $t'$.
Given a subset $E$ of the leaf labels of a tree $T$,
we write $T|_E$ for the subtree of $T$ that is spanned by
the leaves labeled with $E$.

\begin{proposition} \label{prop:inducedtrees}
Let $T$ and $T^{\prime}$ be the trees corresponding to disjoint $(-1)$-curves on
a cubic surface $X$, and  $E$ the set of five leaf labels common to $T$ and $T^{\prime}$.
Then $t = T^{\prime}|_E$  and 
 $t^{\prime} = T|_E$.
\end{proposition}

\begin{proof}
The five lines that meet two disjoint $(-1)$-curves $C$ and $C'$
define five points on $C$ and five tritangent planes containing $C'$.
The cross-ratios among the former are equal to
the cross-ratios among the latter modulo $C'$, see \cite[Section 4]{Naruki}.
The proposition follows because the metric trees can be derived from the
valuations of all the various cross ratios.
\end{proof}

Proposition \ref{prop:inducedtrees} suggests a combinatorial method
for recovering the entire arrangement of $27$ trees on ${\rm trop}(X^0)$
from a single tree $T$. Namely, for any tree $T^{\prime}$ that is disjoint from $T$,
 we can recover both $t^{\prime}$ and $T'|_E$. Moreover, for any of the $10$  trees $T_i$ 
that are disjoint from both $T$ and $T^{\prime}$,
  with labels $E_i$ common with $T$, we can determine $   T^{\prime}|_{E_i}$ as well.
   Then $T'$ is an amalgamation of $t'$, $T'|_{E}$,
   and the $10$ subtrees $   T^{\prime}|_{E_i}$.
  This amalgamation process is reminiscent of 
  a tree building algorithm in phylogenetics known as
   {\em quartet puzzling} \cite{BrySte}.

\smallskip

\begin{figure}[b]
\centering
\vspace{-0.15in}
\includegraphics[scale=0.8]{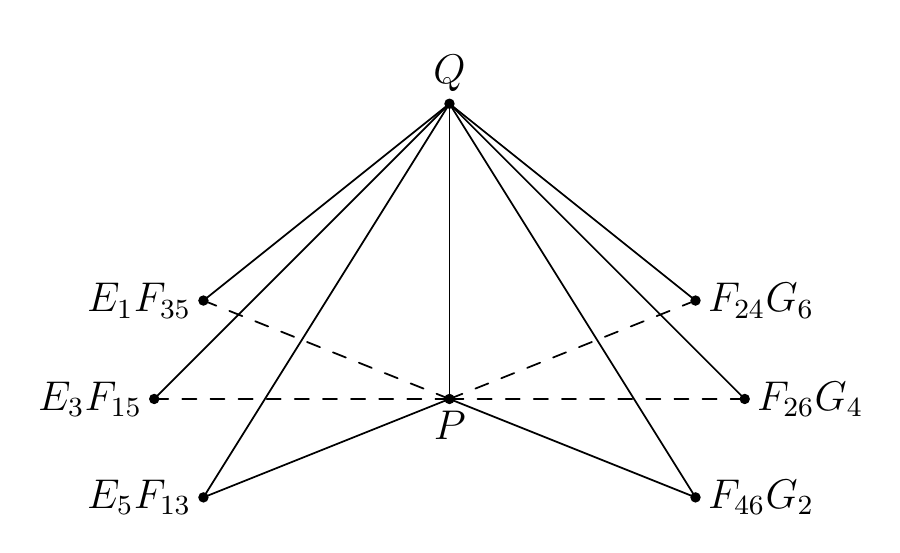}
\vspace{-0.2in}
\caption{The bounded complex of  the tropical cubic surface of type (a)
\label{figure:finite_part_a}}
\end{figure}

\begin{figure}
\centering
\vspace{-0.15in}
\includegraphics[scale=0.7]{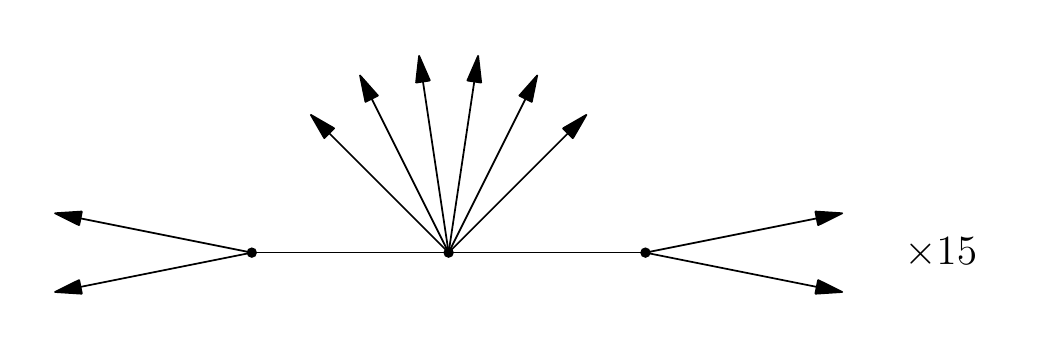}
\includegraphics[scale=0.7]{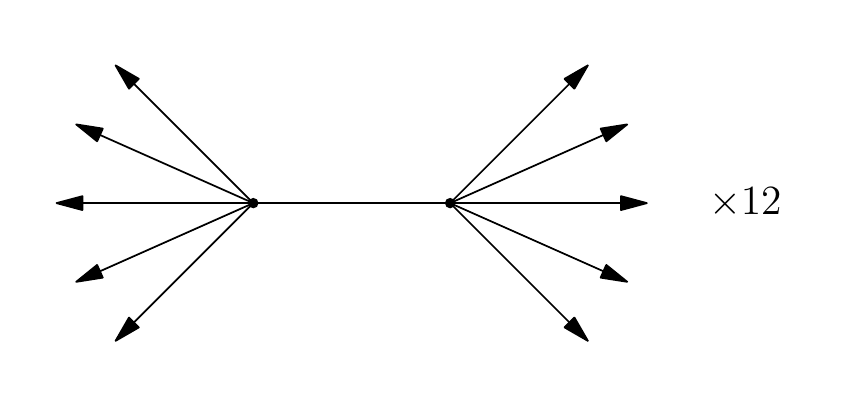}
\vspace{-0.15in}
\caption{The $27$ trees on the tropical cubic surface of  type (a)
\label{figure:tree_a}}
\end{figure}

We next examine tropical cubic surfaces of
 non-generic types. These surfaces are obtained from
 non-generic fibers of the vertical map on the right in
 (\ref{equation:fan_commutative_diagram}).
 We use the subdivision of the 
 Naruki fan ${\rm trop}(\mathcal{Y}^0)$
  described in Lemma~\ref{lem:barycentric}.
  There are five types of rays in this subdivision. We label them (a), (b), (${\rm a}_2$), 
  (${\rm a}_3$), (${\rm a}_4$). 
   A ray of type (${\rm a}_k$) is a positive linear combination of $k$ rays of type (a).
  The new rays (${\rm a}_2$),   (${\rm a}_3$), (${\rm a}_4$) form the barycentric subdivision of an
          (aaaa) cone.
          With this,
          the maximal cones in the subdivided Naruki fan are called
(${\rm aa}_2 {\rm a}_3 {\rm a}_4$) and (${\rm aa}_2 {\rm a}_3$b).
They are known as the generic types (aaaa) and (aaab) in the previous sections.
A list of all $24$ cones, up to symmetry, is presented in
          the first column of Table  \ref{tab:cubicsurf}.

The fiber of ${\rm trop}(\mathcal{G}^0)
\rightarrow {\rm trop}(\mathcal{Y}^0)$ over any point in the interior of
a maximal cone is a tropical cubic surface.
However, some special fibers
have dimension $3$. Such fibers
contain infinitely many tropical cubic surfaces,
including those with
  Eckhart points.  
   Removing such Eckhart points is a key issue in \cite{HKT}.
  We do this by considering  the {\it stable fiber}, 
  {\it i.e.~}the limit of the generic fibers obtained by perturbing the base point by an infinitesimal. Alternatively, 
  the tree arrangement of the 
   stable fiber is found by setting
 some edge lengths to $0$
 in Remark~\ref{rem:paraxxx}.
  We computed representatives for all stable fibers.
  Our results are shown in Table~\ref{tab:cubicsurf}.

\begin{figure}[h]
\centering
\vspace{-0.15in}
\includegraphics[scale=0.7]{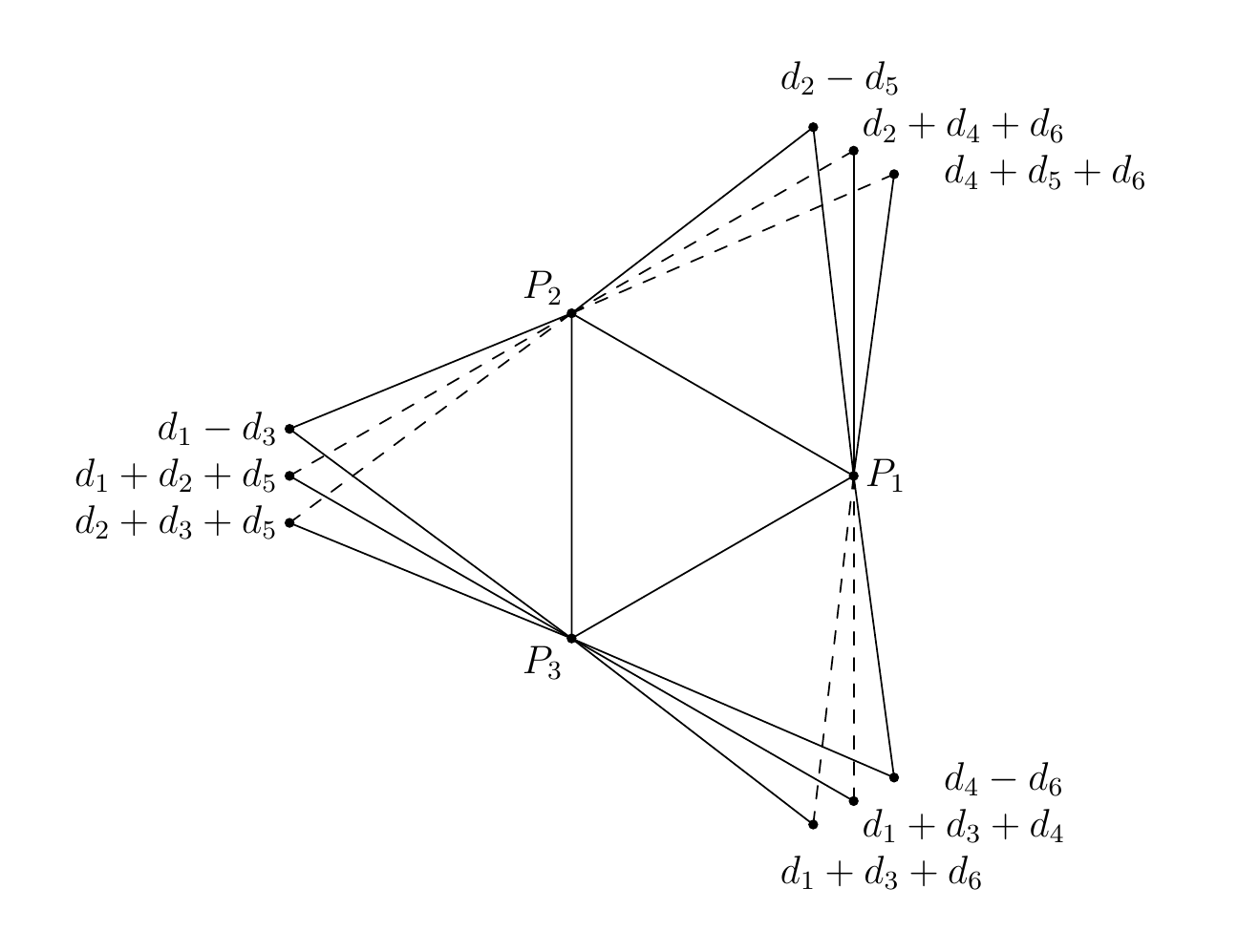}
\vspace{-0.27in}
\caption{The bounded complex of the tropical cubic surface of type (b)
\label{figure:finite_part_b}}
\end{figure}

\begin{figure}[b]
\centering
\vspace{-0.15in}
\includegraphics[scale=0.7]{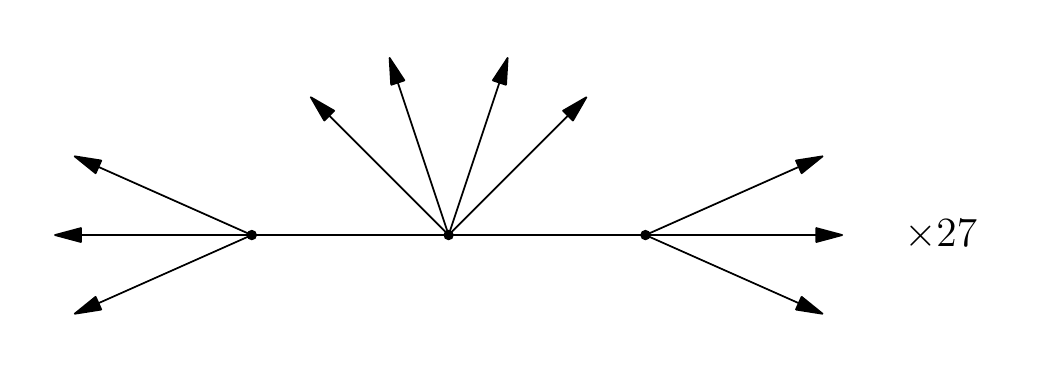}
\vspace{-0.17in}
\caption{The $27$ trees on the tropical cubic surface of type (b)
\label{figure:tree_b}}
\end{figure}

\begin{table}[h]
\centering
\begin{tabular}{c|c|c|cc|cccc}
Type & \#cones in moduli & Vertices & Edges & Rays & Triangles & Squares & Flaps & Cones \\
\hline
$0$ & 1 & 1 & 0 & 27 & 0 & 0 & 0 & 135 \\
\hline
(a) & 36 & 8 & 13 & 69 & 6 & 0 & 42 & 135 \\
(${\rm a}_2$) & 270 & 20 & 37 & 108 & 14 & 4 & 81 & 135 \\
(${\rm a}_3$) & 540 & 37 & 72 & 144 & 24 & 12 & 117 & 135 \\
(${\rm a}_4$) & 1620 & 59 & 118 & 177 & 36 & 24 & 150 & 135 \\
(b) & 40 & 12 & 21 & 81 & 10 & 0 & 54 & 135 \\
\hline
(${\rm a} {\rm a}_2$) & 540 & 23 & 42 & 114 & 13 & 7 & 87 & 135 \\
(${\rm aa}_3$) & 1620 & 43 & 82 & 156 & 22 & 18 & 129 & 135 \\
(${\rm aa}_4$) & 540 & 68 & 133 & 195 & 33 & 33 & 168 & 135 \\
(${\rm a}_2 {\rm a}_3$) & 1620 & 43 & 82 & 156 & 22 & 18 & 129 & 135 \\
(${\rm a}_2 {\rm a}_4$) & 810 & 71 & 138 & 201 & 32 & 36 & 174 & 135 \\
(${\rm a}_3 {\rm a}_4$) & 540 & 68 & 133 & 195 & 33 & 33 & 168 & 135 \\
(ab) & 360 & 26 & 48 & 123 & 16 & 7 & 96 & 135 \\
(${\rm a}_2$b) & 1080 & 45 & 86 & 162 & 24 & 18 & 135 & 135 \\
(${\rm a}_3$b) & 1080 & 69 & 135 & 198 & 34 & 33 & 171 & 135 \\
\hline
(${\rm aa}_2 {\rm a}_3$) & 3240 & 46 & 87 & 162 & 21 & 21 & 135 & 135 \\
(${\rm aa}_2 {\rm a}_4$) & 1620 & 74 & 143 & 207 & 31 & 39 & 180 & 135 \\
(${\rm aa}_3 {\rm a}_4$) & 1620 & 74 & 143 & 207 & 31 & 39 & 180 & 135 \\
(${\rm a}_2 {\rm a}_3 {\rm a}_4$) & 1620 & 74 & 143 & 207 & 31 & 39 & 180 & 135 \\
(${\rm aa}_2$b) & 2160 & 48 & 91 & 168 & 23 & 21 & 141 & 135 \\
(${\rm aa}_3$b) & 3240 & 75 & 145 & 210 & 32 & 39 & 183 & 135 \\
(${\rm a}_2 {\rm a}_3$b) & 3240 & 75 & 145 & 210 & 32 & 39 & 183 & 135 \\
\hline
(${\rm aa}_2 {\rm a}_3 {\rm a}_4$) & 3240 & 77 & 148 & 213 & 30 & 42 & 186 & 135 \\
(${\rm aa}_2 {\rm a}_3$b) & 6480 & 78 & 150 & 216 & 31 & 42 & 189 & 135
\end{tabular}
\caption{All combinatorial types of tropical cubic surfaces
\label{tab:cubicsurf}}
\end{table}

We  explain the two simplest non-trivial cases. The $36$ type (a) rays in the Naruki fan are in bijection with the $36$ positive roots of $\mathrm{E}_6$. Figure \ref{figure:finite_part_a} shows the bounded cells in the stable fiber over the (a) ray corresponding to root $r=d_1+d_3+d_5$.
It consists of six triangles sharing a common edge.
The two shared vertices are labeled by $P$ and $Q$.
  Recall the identification of the roots of $\mathrm{E}_6$ involving $d_7$ with the $27$ $(-1)$-curves from (\ref{eq:dtoEFG}). 
Then, considering  $E_i$, $F_{ij}$ and $G_i$ as roots
of $\mathrm{E}_6$, exactly $15$ of them are orthogonal to $r$. The other $12$ roots are
\begin{equation}\label{equation:non_orthogonal_roots}
E_1,F_{35};\quad{}E_3,F_{15};\quad{}E_5,F_{13};\quad{}F_{24},G_6;\quad{}F_{26},G_4;\quad{}F_{46},G_2.
\end{equation}
These form  a {\em Schl\"afli double six}.
The $36$ double six configurations on a cubic surface
are in bijection with the  $36$ positive roots of $\mathrm{E}_6$. 
Each of the six pairs forms an $\mathrm{A}_2$ subroot system with $d_1+d_3+d_5$. 
The non-shared vertices in the (a) surface 
are labeled by these pairs.

The $12$ rays labeled by 
 (\ref{equation:non_orthogonal_roots}) emanate from $Q$,
 and the other $15$ rays emanate from $P$.
  Each other vertex has $7$ outgoing rays, namely its labels in
   Figure
 \ref{figure:finite_part_a} and the $5$ roots  orthogonal to both of these.
Figure \ref{figure:tree_a}
shows the resulting $27=12+15$ trees at infinity.


The $40$ type (b) rays in the Naruki fan are in bijection with the type $\mathrm{A}_2^{\times{}3}$ subroot systems in $\mathrm{E}_6$. Figure \ref{figure:finite_part_b} illustrates the stable fiber over a point lying on the ray corresponding to
\begin{equation}\label{equation:type_a23_subroot_system}
\begin{array}{c}
d_1-d_3, \,\, d_1+d_2+d_5, \,\, d_2+d_3+d_5, \\
d_2-d_5, \,\, d_2+d_4+d_6, \,\, d_4+d_5+d_6,\\
d_4-d_6, \,\, d_1+d_3+d_4, \,\, d_1+d_3+d_6.
\end{array}
\end{equation}
This is the union of three type $\mathrm{A}_2$ subroot systems that are pairwise orthogonal.
The bounded complex consists of $10$ triangles. The central triangle $P_1P_2P_3$ 
has $3$ other triangles attached to each edge. The $9$ pendant vertices are labeled with the roots in (\ref{equation:type_a23_subroot_system}). The $3$ vertices in the triangles attached to the same edge are labeled with $3$ roots in a type $\mathrm{A}_2$ subroot system.

Each of $P_1$,$P_2$ and $P_3$ is connected with $9$ rays, labeled with the roots in $\mathrm{E}_7\backslash{}\mathrm{E}_6$ that are orthogonal to a type $\mathrm{A}_2$ subroot system in (\ref{equation:type_a23_subroot_system}). Each of the other vertices is connected with $6$ rays. The labels of these rays are the roots in $\mathrm{E}_7\backslash{}\mathrm{E}_6$ that are orthogonal to the label of that vertex but are not orthogonal to the other two vertices in the same group.

All of the $27$ trees are isomorphic, as shown in Figure \ref{figure:tree_b}. In each tree, the $10$ leaves are partitioned into $10=4+3+3$, by orthogonality with the type $\mathrm{A}_2$ subroot systems in (\ref{equation:type_a23_subroot_system}). 
 The bounded part of the tree is connected by two flaps to two edges 
containing the same $P_i$.

\smallskip

We close this paper with a brief discussion of
open questions and future directions.
One obvious question is whether our construction
can be extended to del Pezzo surfaces of 
degree $d = 2$ and $d=1$. In principle,
this should be possible, but the complexity
of the algebraic and combinatorial computations
 will be very high. In particular,
the analogues of Theorem \ref{thm:modify}
for $7$ and $8$ points in $\TP^2$ are likely to 
require rather complicated genericity hypotheses.

For $d=4$, we were able compute the Naruki fan
${\rm trop}(\mathcal{Y}^0)$ without any prior knowledge, by just applying 
the software {\tt gfan} to the $45$ trinomials in Proposition
 \ref{prop:IXdeg4}. We believe that the same will work
  for $d=3$, and that even the
  tropical basis property
 \cite[\S 2.6]{MacStu} will hold:
 
 \begin{conjecture} \label{conj:tropbasis}
The $270$ trinomial relations listed in Proposition
  \ref{prop:IXdeg3}
form a tropical basis.
\end{conjecture}

This paper did not consider 
embeddings of del Pezzo surfaces
into projective spaces. However, 
it would be very interesting to 
study these via the results obtained here.
For cubic surfaces in $\PP^3$, we should see
a shadow of Table \ref{tab:cubicsurf} in  $\TP^3$.
Likewise, for complete intersections of
two quadrics in $\PP^4$, we should see a shadow
of Figures \ref{fig:clebsch} and \ref{fig:clebsch2} in $\TP^4$.
One approach is to start with the following 
tropical modifications of the ambient spaces $\TP^3$ resp.~$\TP^4$.
Consider a graded component in
(\ref{eq:coxring}) with $\mathcal{L}$  very ample.
Let $N+1$ be the number of monomials
in $E_i,F_{ij}, G_k$ that lie in $H^0(X,\mathcal{L})$.
The map given by these monomials
embeds $X$ into a linear subspace  of $\PP^{N}$.
The corresponding tropical surfaces in $\TP^{N}$
should be isomorphic to the tropical del Pezzo surfaces constructed here.
In particular, if $\mathcal{L} = -K$ is the anticanonical bundle,
then the subspace has dimension $d$, and the ambient
dimensions are $N = 44$ for $d=3$,
and $N= 39$ for $d=4$.
In the former case, the $45$ monomials
(like $E_1 F_{12} G_{2}$ or $F_{12} F_{34} F_{56}$)
correspond to  Eckhart triangles.
In the latter case, the $40$ monomials
(like  $E_1 E_2 F_{12} G$ or $E_1 F_{12} F_{13}  F_{45}$)
are those of degree $(4,2,2,2,2,2)$ in the
grading~(\ref{eq:D5grading}). 
The tropicalizations of  these {\em combinatorial anticanonical embeddings},
$X \subset \PP^3 \subset \PP^{44}$ for $d=3$ and 
$X \subset \PP^4 \subset \PP^{39}$ for $d=4$,
should agree with our surfaces here.
This  will help in resolving remaining issues surrounding the excess of 
lines in tropical cubic surfaces. Examples of the superabundance of tropical lines 
 on generic smooth tropical cubic 
hypersurfaces were first found by Vigeland~\cite{Vig} and these examples were later considered in 
\cite{BK} and  \cite{BS}. 

One last consideration concerns cubic surfaces defined over $\R$.
 A cubic surface  equipped with a real structure induces another involution on the $27$ metric 
 trees corresponding to real  $(-1)$-curves.
These trees already come partitioned by combinatorial type, depending on the type of tropical cubic surface. 
One could ask which trees can result from real lines, 
and whether the tree arrangement reveals
 Segre's partition of real lines on cubic surfaces
 into hyperbolic and elliptic types \cite{Segre}.
   For example, for the (aaaa) and (aaab) types,
 if the involution on the trees from the real structure is the trivial one, then the trees with combinatorial type occurring exactly three times 
always correspond to hyperbolic real lines.

\bigskip
\bigskip

\noindent
{\bf Acknowledgements.}\\
This project started during the 2013 program on
{\em Tropical Geometry and Topology} at 
the Max-Planck Institut f\"ur Mathematik
in Bonn, with Kristin Shaw and Bernd Sturmfels in residence.
Qingchun Ren and Bernd Sturmfels
were supported by NSF grant DMS-0968882. Kristin Shaw had
support from the Alexander von Humboldt Foundation in the form of a Postdoctoral Research Fellowship. 
We are grateful to Maria Angelica Cueto, Anand Deopurkar 
and also an anonymous referee for helping us to improve this paper.

\smallskip
\bigskip

\footnotesize
\noindent {\bf Authors' addresses:}

\smallskip

\noindent Qingchun Ren,  Google Inc, Mountain View, CA 94043,
USA, {\tt qingchun.ren@gmail.com}
\smallskip

\noindent Kristin Shaw,  Technische Universit\"at
Berlin, MA 6-2, 10623 Berlin, Germany,
{\tt shaw@math.tu-berlin.de}

\smallskip

\noindent Bernd Sturmfels,  University of California, Berkeley, CA 94720-3840,
USA, {\tt bernd@berkeley.edu}

\end{document}